\documentclass[11pt]{article}
\usepackage[english]{babel}
\usepackage[utf8]{inputenc}
\usepackage{amsmath,amsfonts,amssymb,amsthm}
\usepackage[letterpaper,top=2cm,bottom=2cm,left=3cm,right=3cm,marginparwidth=1.75cm]{geometry}
\usepackage{graphicx}
\usepackage[colorlinks=true, allcolors=blue]{hyperref}
\usepackage{enumitem}
\usepackage{bm}
\usepackage{cleveref}
\usepackage[numbers,sort&compress]{natbib}
\usepackage[table]{xcolor}
\usepackage{longtable,diagbox}
\usepackage{xcolor}
\newtheorem{theorem}{Theorem}[section]

\newtheorem{lemma}[theorem]{Lemma}
\newtheorem{proposition}[theorem]{Proposition}

\theoremstyle{definition}
\newtheorem{definition}[theorem]{Definition}

\newtheorem{assumption}[theorem]{Assumption}

\usepackage{authblk}
\usepackage{mathrsfs}
\usepackage[dvipsnames]{xcolor}
\usepackage{cancel}
\usepackage{tikz}
\usetikzlibrary{shapes, positioning, calc}
\usetikzlibrary{positioning, arrows.meta}
\usepackage{subcaption}
\usepackage{lineno}
\newtheorem*{variance1}{Principle of Optimal Variance Reduction (V1)}
\newtheorem*{variance2}{Principle of Optimal Variance Reduction (V2)}

\DeclareMathOperator*{\argmin}{arg\,min}
\providecommand{\keywords}[1]
{
  \small	
  \textbf{\textit{Keywords---}} #1
}

\title{Data-Adaptive Learning of Dynamical Systems by Matching Transfer Operators and Invariant Measures}
\author[1]{Yinong Huang}
\author[2]{Jonah Botvinick-Greenhouse}
\author[3]{Yunan Yang}
\affil[1]{Department of Mathematics, North Carolina State University, Raleigh, NC}
\affil[2]{Center for Applied Mathematics, Cornell University, Ithaca, NY}
\affil[3]{Department of Mathematics, Cornell University, Ithaca, NY}

\date{}
\begin{document}

\maketitle 

\begin{abstract}

Trajectory-based learning of dynamical systems is often fragile in the presence of noise, chaos, or sparse observations, as small pointwise errors can rapidly amplify. We introduce a transition-statistics approach to system identification that learns dynamics by matching the induced motion of probability mass across a data-adaptive mesh. Given trajectory data, we build an unstructured partition of state space and approximate the Perron--Frobenius operator with a regularized Ulam transition matrix. We replace hard cell indicators with continuous, piecewise-smooth partition-of-unity weights, yielding a Markov matrix supporting gradient-based optimization with respect to the parameters of a learned vector field. This enables two related training objectives: matching invariant measures through the stationary eigenvectors of the transition matrices, and matching the full transition matrices to capture transport between regions of state space. Numerical experiments on Lorenz-63, Lorenz-96, and a reduced-order NOAA sea surface temperature forecasting problem show that transition-statistics matching gives more reliable long-time dynamics than pointwise trajectory matching, particularly under measurement noise and sparse sampling. The approach provides a robust operator-theoretic alternative to trajectory-level losses for learning chaotic and partially observed dynamical systems.
\end{abstract}


\section{Introduction}

Data-driven dynamical system identification seeks to infer an evolution law from observed trajectories. This task arises in weather and climate prediction, biological modeling, engineering systems, and reduced-order modeling of complex physical processes. A common approach is to fit a vector field or flow map by comparing simulated trajectories with observed ones. For example, sparse regression methods such as SINDy estimate candidate terms in the governing equations from trajectory data \cite{brunton2016discovering}, while neural differential equation and shooting-based methods parameterize the vector field and optimize a pointwise trajectory loss \cite{DBLP:journals/corr/abs-1806-07366, linot2023stabilized}.

Such trajectory-based, or Lagrangian, objectives can be effective when the data are accurate and sufficiently well sampled. However, they become fragile under noise, sparsity, and chaos which are the regimes considered in this paper. Measurement noise corrupts derivative estimates and can cause a learned model to overfit local fluctuations. Slow sampling makes the one-step map harder to infer and can introduce spurious minima in the optimization landscape. Chaotic dynamics introduce an additional difficulty: even when a model has the correct qualitative behavior, small errors in the vector field or initial condition may lead to rapid pointwise separation of trajectories. Thus, long-time pointwise trajectory agreement is often too stringent a criterion for learning chaotic systems. 

An alternative is to compare statistical quantities that are stable under the long-time dynamics. Invariant measures provide one such description: rather than recording the precise location of a trajectory at each time, they describe the distribution of states visited asymptotically by the system. For chaotic systems, this perspective is natural because nearby trajectories may separate while still sampling the same invariant measure. This observation has motivated a class of Eulerian system identification frameworks in which the learned dynamics are trained to reproduce global statistical information extracted from observed data~\cite{jiang2023training, schiff2024dyslim, botvinick2025invariant2,giorgini2026conditional,melo2026learning,greve2019data}. This motivates learning objectives that compare the induced transport of probability mass over finite time, rather than comparing individual trajectories pointwise.

In this work, we focus on a finite-dimensional operator-theoretic representation of these statistics. For a fixed lag time, the Perron--Frobenius operator can be approximated on a partition of state space by a Markov transition matrix whose entries describe the probabilities of moving between cells over that lag. The invariant measure is then represented by a dominant eigenvector of this matrix. Matching invariant measures therefore uses only part of the information contained in the transition matrix. In contrast, matching the full Markov matrix compares the transition probabilities themselves and can capture how mass moves through the state space, not just where it accumulates in the long-time limit. %
While invariant-measure matching has recently been used for system identification, directly matching finite-dimensional Markov transition matrices appears to have received much less attention as a training objective for learning dynamical systems.

The main computational challenge in matching either the invariant measure or transition statistics is that Markov matrix approximations are typically constructed from a partition of the state space. Uniform grids are simple in low dimensions, but their cost grows rapidly with the ambient dimension and they allocate many cells to regions never visited by the data. This is especially wasteful for dissipative or chaotic systems whose trajectories concentrate near an attractor of much smaller intrinsic dimension than the surrounding state space. Motivated by this phenomenon, several works in the literature have focused on data-adaptive meshes for approximating the PFO in forward inference tasks related to the approximation of invariant measures and computation of transition times \cite{dellnitz1998adaptive,murray2004optimal,wehlitz2026data,bittracher2018data,bonner2023improving}. 

However, in the context of dynamical system learning, where one aligns model-simulated and observed invariant measures, uniform discretizations are typical and data-adaptive meshes have received little attention \cite{greve2019data,yang2023optimal,botvinick2023learning}. In this work, we  construct a data-adaptive mesh using a Voronoi partition based on $k$-means clustering of the observed trajectory. This mesh serves as the foundation for constructing observed and model-simulated transition matrices, thereby enabling mesh-based Eulerian approaches to system identification to scale to high-dimensional applications. We also analyze the error of the empirical transition matrix and derive a variance-reduction principle which suggests that sample counts should be evenly distributed across mesh cells, thereby motivating our data-adaptive clustering.

The second computational challenge is differentiability. Classical Ulam matrices use hard cell indicators, so small changes in the learned vector field often do not change the assigned cell of a predicted point. This produces vanishing or unstable gradients, making the transition matrix difficult to use inside a neural-network training loop. We therefore use a regularized Ulam construction in which hard indicators are replaced by continuous nonnegative weights forming a partition of unity. The resulting Markov matrix remains stochastic, but its entries vary continuously with the predicted next states and hence with the parameters of the learned vector field. We prove convergence of the regularized Markov matrix in an iterated limit where the regularization parameter tends to zero, before the mesh is refined (Theorem \ref{thm:convergence}). 

The present paper develops and tests this differentiable Markov matrix framework for learning dynamics from data. The regularized data-adaptive  approximation builds on an earlier technical note by one of the authors \cite{botvinickgreenhouse2025invariant}, but the emphasis here is different: we use this approximation as the basis for direct Markov matrix matching in system identification applications, compare it systematically with pointwise trajectory matching, and evaluate its behavior under noisy and sparsely sampled data. In addition to synthetic Lorenz-63 and Lorenz-96 experiments, we apply the method to a reduced-order forecasting problem for NOAA sea surface temperature data. The main contributions of this paper are as follows.

\begin{itemize}
    \item \textbf{Differentiable transition-statistics objectives:}
    We formulate invariant-measure and Markov-matrix matching objectives for system identification using regularized Ulam approximations on data-adaptive meshes. The Markov-matrix objective compares finite-time transition probabilities between mesh cells, retaining information that is lost when only stationary distributions are matched.

    \item \textbf{Regularized Ulam approximation and convergence:}
    We replace hard cell indicators by continuous partition-of-unity weights, producing stochastic transition matrices that are differentiable with respect to the parameters of a learned vector field. We prove convergence of the regularized approximation in an iterated limit as the regularization is removed and the mesh is refined.

    \item \textbf{Data-adaptive mesh construction and sampling error:}
    We analyze the approximation error in empirical transition-matrix estimation and derive a variance-reduction principle for partition design. This motivates data-adaptive meshes that allocate resolution to regions visited by the observed dynamics.

    \item \textbf{Numerical evaluation under noise, sparsity, and chaos:}
    We compare trajectory matching, invariant-measure matching, and Markov-matrix matching on Lorenz-63, Lorenz-96, and a reduced-order NOAA sea surface temperature forecasting problem. The experiments show that transition-statistics objectives give more stable long-time reconstructions under noisy and sparsely sampled data.
\end{itemize}

The paper is organized as follows. \Cref{sec:background} reviews the PFO, invariant measures, and finite-dimensional Markov matrix approximations. \Cref{sec:mesh} describes the data-adaptive regularized Ulam construction used to build differentiable transition matrices. Section \ref{subsec:galerkin} presents theoretical results concerning the convergence of our regularized approximation and error analysis of the finite-sample approximation.  Numerical experiments are presented in \Cref{sec:numerics}, including mesh comparisons, invariant measure matching, Markov matrix matching, and the NOAA sea surface temperature example. Conclusions follow in \Cref{sec:conclusion}.

\section{Background}\label{sec:background}

In this section, we review the background needed for our data-adaptive approximation of the PFO and for the two system identification objectives used in this paper. We begin by relating continuous-time dynamics to discrete-time maps, then define the PFO, invariant measures, and their finite-dimensional Markov matrix approximations.

\subsection{Dynamical Systems and Time-\texorpdfstring{$\Delta t$}{Delta t} Maps}

Let $X$ denote a state space with smooth structure, such as a Riemannian manifold or $\mathbb{R}^d$. We consider continuous-time dynamical systems of the form  $\dot{x} = v(x),$ where $v$ is the vector field on $X$ that we seek to learn from observed data. For an initial condition $x\in X$, we denote by $f_t(x)$ the solution of $\dot{x} = v(x)$ at time $t$. Thus, $f_t:X\to X$ is the time-$t$ flow map generated by $v$. 

We assume that our data are observed at uniformly spaced discrete times, and thus we work directly with the time-$\Delta t$ flow map $T:=f_{\Delta t}:X\to X$. The observed trajectory, which is used for system identification, can then be written as $\{x_{k}\}_{k=0}^{N}$ where $x_{k+1}=T(x_k)$ for $k = 0,1,\dots, N-1$, up to measurement noise. Although the underlying dynamics are continuous in time, the loss functions defined in this work are formulated in terms of the discrete map $T$. This viewpoint is convenient because the PFO and its Markov matrix approximations describe how probability distributions evolve under a single application of the map $T$. In the subsequent discussion, we shift our focus to an arbitrary discrete map $T:X\to X$ and allow $X$ to be a general measurable space.

\subsection{The Perron--Frobenius Operator (PFO)}\label{subsec:perron}

Let $(X,\mathscr{B},\mu)$ be a measure space and let $T:X\to X$ be a measurable map. The PFO, also called the transfer operator, describes the evolution of densities under the dynamics. While the state-space map $T$ may be nonlinear, the PFO is linear on densities. We assume that $T$ is non-singular with respect to $\mu$, meaning that
\[
    \mu(B)=0
    \quad \Longrightarrow \quad
    \mu(T^{-1}(B))=0,
\]
for all $B\in\mathscr{B}$. The PFO is then defined as follows.

\begin{definition}[Perron--Frobenius operator~\cite{lasota2013chaos}]\label{def:PFO}
Let $T:X\to X$ be non-singular with respect to $\mu$. The Perron--Frobenius operator $\mathcal{T}:L^1_\mu(X)\to L^1_\mu(X)$
is the unique linear operator satisfying
\begin{equation}\label{eq:PFO_definition}
    \int_B \mathcal{T}f\,d\mu
    =
    \int_{T^{-1}(B)} f\,d\mu,
    \qquad
    \forall B\in\mathscr{B},
    \qquad
    \forall f\in L^1_\mu(X).
\end{equation}
\end{definition}

The operator $\mathcal{T}$ is a Markov operator, meaning that it preserves nonnegativity and total mass. In particular, if $f\geq 0$ and $\int_X f\,d\mu=1$, then $\mathcal{T}f\geq 0$ and
\[
    \int_X \mathcal{T}f\,d\mu = 1.
\]
Thus, $\mathcal{T}$ maps probability densities to probability densities. This property is central to our finite-dimensional approximation which preserves the Markov structure of the continuous operator.

\subsection{Invariant Measures}

An invariant measure describes the long-time statistical behavior of a dynamical system. Let $\mathscr{P}(X)$ denote the set of Borel probability measures on $X$. If $\nu\in\mathscr{P}(X)$ and $T:X\to X$ is measurable, the pushforward measure $T_\#\nu$ is defined by
\[
    (T_\#\nu)(B) = \nu(T^{-1}(B)),
    \qquad \forall B\in\mathscr{B}.
\]

\begin{definition}[Invariant measure]
 A probability measure $\nu\in\mathscr{P}(X)$ is said to be $T$-invariant if $T_\#\nu = \nu.$
\end{definition}

If $\nu$ is absolutely continuous with respect to the reference measure $\mu$, then $d\nu=\rho\,d\mu$ for some density $\rho\in L^1_\mu(X)$. In this case, invariance of $\nu$ is equivalent to the fixed-point relation
$\mathcal{T}\rho = \rho$, where $\mathcal{T}$ is defined in \eqref{eq:PFO_definition}. Thus, invariant densities are eigenfunctions of the PFO associated with eigenvalue $1$. 

Under suitable ergodicity assumptions, one can estimate the invariant measure, or its coarse-grained cell probabilities, from a single long trajectory. Many computational approaches for approximating invariant measures exist, including histogram estimation \cite{greve2019data, yang2023optimal}, generative models \cite{giorgini2025score}, and Ulam's method \cite{ding1996finite, Sch_tte_2016,froyland2007detection}. In this work, we focus on the latter approach, which estimates the invariant measure as a dominant eigenvector of a Markov transition matrix approximating the continuous PFO. 

\subsection{Ulam's Method}\label{subsec:ulam}

The PFO defined in~\eqref{eq:PFO_definition} is infinite-dimensional, so it must be approximated before it can be used computationally. A standard approach is Ulam's method, which partitions the state space into cells and approximates the transition probabilities between these cells. Let $\{C_i\}_{i=1}^n$ be a partition of $X$. The associated Markov matrix $M\in\mathbb{R}^{n\times n}$ has entries
\begin{equation}\label{eq:ulam_background}
    M_{ij}
    =
    \frac{\mu(C_i\cap T^{-1}(C_j))}{\mu(C_i)}.
\end{equation}
Equivalently, $M_{ij}$ is the conditional probability that a point starting in cell $C_i$ moves to cell $C_j$ after one application of $T$. Thus, each row of $M$ sums to one, and $M$ gives a finite-dimensional approximation of the PFO.  Moreover, the dominant left eigenvector of $M$ approximates the underlying invariant measure over the partition cells. The convergence of $M$ and its dominant eigenvector under partition refinement has been previously studied~\cite{li1976finite,ding1996finite,froyland1996estimating}.

When only trajectory data is available, the entries of $M$ can be estimated empirically using Monte Carlo integration or ergodic averages. In particular, if $\{x_k^i\}_{k=1}^{N_i}$ are the observed samples lying in cell $C_i$, then
\begin{equation}\label{eq:empirical_markov_background}
    \widehat{M}_{ij}
    =
    \frac{1}{N_i}
    \sum_{k=1}^{N_i}
    \chi_{C_j}(T(x_k^i)).
\end{equation}
In the data-driven setting, $T(x_k^i)$ in \eqref{eq:empirical_markov_background} is replaced by the observed next state. For a parameterized model $v_\theta$, it is replaced by the one-step prediction obtained by integrating the learned dynamics over a time $\Delta t$.

\section{Matching Transition Matrices and Invariant Measures}
\label{sec:mesh}

In this section, we provide a detailed description of our framework for learning dynamical systems through the comparison of transition statistics and invariant measures. Section \ref{subsec:mesh} outlines our framework for data-adaptive mesh construction over which we define transition matrices. In Section \ref{subsec:diff_trans}, we discuss a differentiable reformulation of Ulam's method which replaces the hard-cell indicators in \eqref{eq:ulam_background} with continuous partition-of-unity weights, thereby enabling stable gradient-based optimization. Sections~\ref{subsec:matrix_comp} and~\ref{subsec:invariant_comp} then introduce the training objectives we consider for the comparison of transition statistics and invariant measures, respectively.

\subsection{Data-Adaptive Partition}\label{subsec:mesh}

Classical discretizations of the PFO often begin by partitioning the state space into a uniform grid. This is convenient in one or two dimensions, but it becomes inefficient as the number of grid cells grows rapidly with dimension; see Table~\ref{tab:uniform_cost} for the cost of increasing the number of cells in a representative uniform-mesh computation. Many cells may lie in regions that are never visited by the observed trajectory. This issue is especially severe for dissipative or chaotic systems whose long-time dynamics concentrate near an attractor.

\begin{table}[h!]
    \centering
    \begin{tabular}{|c|c|c|}
    \hline
    Number of cells     & Wall-clock time (s) & Memory usage (MB)  \\
    \hline
    $\approx 10^2$ & 0.07 & 0.07 \\
    $\approx 10^3$ & 0.08 & 0.39 \\
    $\approx 10^4$ & 0.12 & 3.33\\
    $\approx 10^5$ & 34.04 &320.23 \\
    \hline
    \end{tabular}
    \caption{Computational cost of approximating the invariant measure of the van der Pol oscillator using uniform meshes and the PDE-based approach in \cite{botvinick2023learning}. The computation is performed using an Intel i7-1165G7 CPU.}
    \label{tab:uniform_cost}
\end{table}

Let $
    \{x_k\}_{k=0}^{N-1}\subseteq \mathbb{R}^d$
denote the observed trajectory samples. We first choose $n$ initial points without replacement from the observed trajectory  uniformly at random. These $n$ points are then used to initialize a $k$-means or MiniBatch $k$-means clustering routine, which results in $n$ cell centers $c_1,\dots,c_n\in\mathbb{R}^d$. Collectively, these centers define a Voronoi partition $\{C_i\}_{i=1}^n$ of $\mathbb{R}^d$ where the $i$-th cell contains  all points  $x\in \mathbb{R}^d$ such that $|x-c_i| < |x-c_j|$ for $j \neq i$, where $|\cdot|$ denotes the Euclidean distance in $\mathbb{R}^d$. To efficiently determine which cell of a Voronoi diagram a given point $x\in \mathbb{R}^d$ belongs in, we use a k-d tree. Visualizations of the Voronoi partition are shown in Figures \ref{fig:weights} and \ref{fig:CatA}.

The resulting partition is data-adaptive: it places cells in regions frequently visited by the trajectory and avoids allocating resolution to regions with little or no data. This construction is particularly useful when the observed dynamics occupy a low-dimensional or highly nonuniform subset of the ambient state space. The quality of the empirical transition matrix depends on how many samples fall in each cell. If a cell contains very few samples, then its outgoing transition probabilities following \eqref{eq:empirical_markov_background} have high variance. In Section \ref{subsec:MC_analysis}, we introduce a variance reduction principle which motivates constructing cells so that the observed samples are distributed evenly across the partition (Proposition \ref{prop:optimal_strat}). While this is approximately obtained using $k$-means clustering, it can be exactly achieved using certain constrained $k$-means clustering routines~\cite{bradley2000constrained}.

\subsection{Differentiable Transition Matrix}\label{subsec:diff_trans}

Throughout, we let $v_\theta$ be a parameterized vector field, where $\theta\in\Theta \subseteq \mathbb{R}^p$, and let $T_{\theta} = f_{\theta,\Delta t}$ denote its time-$\Delta t$ flow map. While the empirical matrix \eqref{eq:empirical_markov_background} is appropriate for estimating transition probabilities from data, the hard indicator functions $\chi_{C_j}$ are not suitable for gradient-based optimization.  In particular, a small change in $\theta$ may not change the cell assignment of a predicted point $T_{\theta}(x)$. Consequently, the entries of the corresponding transition matrix can be locally constant as functions of $\theta$, leading to vanishing gradients.

\begin{figure}[h]
    \centering
    \includegraphics[width=\textwidth]{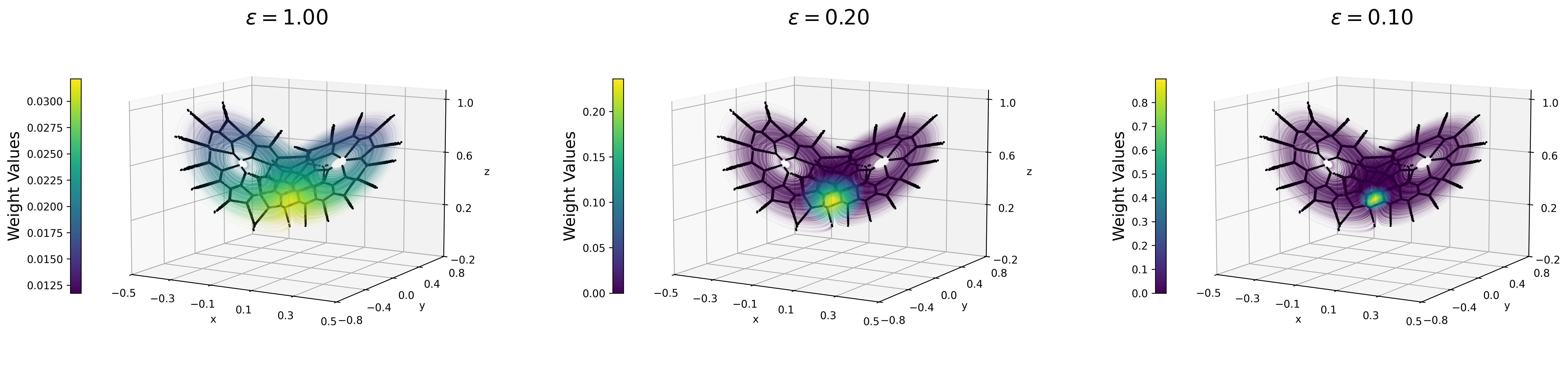}
    \caption{Visualizations of a continuous weighting function $\psi_j^{(\varepsilon)}$ following~\eqref{eq:smooth_weights2} for the Lorenz-63 system \eqref{eq:lorenz63} in a normalized coordinate frame. The black lines are the edges of the Voronoi cells. As $\varepsilon$ decreases, the weighting function approaches the characteristic function.}
    \label{fig:weights}
\end{figure}

\subsubsection{Partition of Unity}
To obtain a differentiable approximation, we replace the characteristic functions $\chi_{C_j}$ by continuous, nonnegative soft-assignment weights
\[
    \psi_j^{(\varepsilon)}:X\to[0,1],
    \qquad j=1,\ldots,n,
\]
forming a partition of unity:
\[
    \sum_{j=1}^n \psi_j^{(\varepsilon)}(x)=1,
    \qquad
    \psi_j^{(\varepsilon)}(x)\geq 0.
\]
The parameter $\varepsilon>0$ controls the localization of the approximation. Smaller values make the weights closer to hard cell indicators, while larger values smooth the transition between neighboring cells. For example, in our numerical experiments we construct the weighting functions as
 \begin{align}
 \psi_j^{(\varepsilon)}(x)
    &=
    \frac{r_j(x;\varepsilon)}
    {\sum_{\ell=1}^n r_\ell(x;\varepsilon)}, \qquad  j = 1,\dots, n, \label{POU}
\end{align}
where each $r_j$ decays as a function of the distance from the cell center $c_j$ following either
\begin{align}
 r_j^{(1)}(x;\varepsilon)
    &=
    \log\!\left(1+\exp\!\left(-\frac{|x-c_j|}{\varepsilon}\right)\right),
    \label{eq:smooth_weights} \\
 r_j^{(2)}(x;\varepsilon)
    &=  \max\left\{0,1-\frac{|x-c_j|}{\varepsilon}\right\}.
    \label{eq:smooth_weights2}
\end{align}
We note that both \eqref{eq:smooth_weights} and \eqref{eq:smooth_weights2} are continuous, piecewise $C^{\infty}$, and that \eqref{eq:smooth_weights} is nonzero everywhere.
To use \eqref{eq:smooth_weights2} in practice, the regularization parameter $\varepsilon > 0$ must be chosen sufficiently large, such that
$\sum_{\ell = 1}^n r_{\ell}(x;\varepsilon) > 0$ for all $x\in \Omega$, where $\Omega \subseteq X$ is the computational domain. For such a choice of $\varepsilon$, both \eqref{eq:smooth_weights} and \eqref{eq:smooth_weights2} can construct  partitions of unity following \eqref{POU}  that mitigate the difficulty of vanishing gradients in Ulam's method. A visualization of the weighting function $\psi_{j}^{(\varepsilon)}$ following \eqref{eq:smooth_weights2} is shown in Figure~\ref{fig:weights}.

\subsubsection{Constructing the Regularized Matrix}
For samples $\{x_k^i\}_{k=1}^{N_i}$ in cell $C_i$, we define the $i$-$j$ entry of the model-induced regularized transition matrix by
\begin{equation}
   \widehat{M}_{ij}^{(\varepsilon)}(v_\theta)
    =
    \frac{1}{N_i}
    \sum_{k=1}^{N_i}
    \psi_j^{(\varepsilon)}
    \bigl(T_{\theta}(x_k^i)\bigr).
    \label{eq:regularized_transition_matrix}
\end{equation}
  Because the weights $\psi_{j}^{(\varepsilon)}$ form a partition of unity, each row of $ \widehat M^{(\varepsilon)}(v_\theta)$ sums to one. Hence the regularized construction preserves the stochastic structure of the Ulam matrix. At the same time, the entries vary continuously with the predicted states and therefore with the parameters $\theta$, allowing the transition matrix to be used for gradient-based optimization.
  
\begin{figure}[h]
    \centering
    \includegraphics[width=\textwidth]{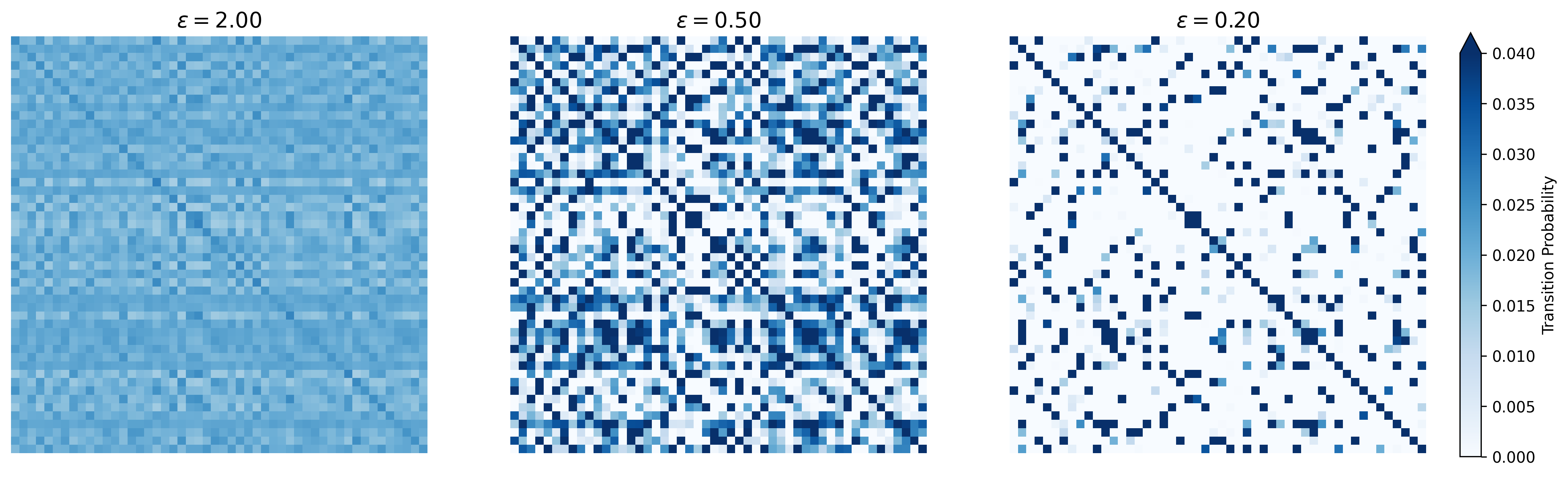}
    \caption{Visualization of the regularized Markov matrices following~\eqref{eq:smooth_weights2} generated with different weighting parameters.}
    \label{fig:weights_mats}
\end{figure}

The same regularized construction can be applied both to the observed data and to the model predictions. From the observed trajectory, we construct a reference matrix $\widehat{M}^{(\varepsilon)}$ by replacing the model prediction $T_{\theta}(x_{k}^i)$ 
in \eqref{eq:regularized_transition_matrix} with the observed next state. Using the same cells and the same regularization $\varepsilon > 0$ in both matrices reduces discretization mismatch between the reference data and the forward model. While we introduced $\varepsilon > 0$ to promote differentiability of \eqref{eq:regularized_transition_matrix}, we also find that larger values of $\varepsilon > 0$ can promote robustness to noise in dynamical system learning tasks. A visualization of the regularized transition matrix is shown in Figure~\ref{fig:weights_mats}.

\subsection{Comparing Markov Transition Matrices}\label{subsec:matrix_comp}

We now introduce our Markov matrix matching objective, in which the model $v_{\theta}$ is trained by comparing the simulated and observed transition matrices. In particular, we attempt to extract the unknown parameters $\theta\in \Theta$ by optimizing the loss 
\begin{equation}\label{eq:background_M_loss}
    \mathcal{J}_{M}(\theta)
    =
    \mathsf{D}\bigl(\widehat{M}^{(\varepsilon)}(v_\theta),M^*\bigr),
\end{equation}
where $\mathsf{D}$ is a discrepancy between stochastic matrices, $\widehat{M}^{(\varepsilon)}(v_\theta)$ is the model-estimated regularized transition matrix following \eqref{eq:regularized_transition_matrix} and $M^*$ is the regularized transition matrix estimated from data. In our numerical experiments, $v_{\theta}$ is parameterized as a neural network and \eqref{eq:background_M_loss} is minimized using gradient-based methods. 

The metric $\mathsf{D}$ on stochastic matrices may be chosen as the Frobenius norm. Since both $\widehat{M}^{(\varepsilon)}(v_\theta)$ and $M^*$ are stochastic matrices, they can also be compared as $n\times n$ images under the Wasserstein distance arising from the theory of optimal transport \cite{villani2021topics}. While we find this improves robustness, it is computationally expensive when $n$ is large. A cheaper alternative is to compare corresponding rows using a discrete Wasserstein distance:
\begin{equation}\label{eq:background_M_loss2}
    \mathcal{J}_{M}(\theta)
    =
    \sum_{i=1}^n W_2\bigl(\widehat{M}_{i,:}^{(\varepsilon)}(v_\theta),M_{i,:}^*\bigr).
\end{equation}
We empirically find that \eqref{eq:background_M_loss2} is both robust to noise and computationally efficient. It effectively balances the pointwise sensitivity of the Frobenius norm with the large computational cost of evaluating the Wasserstein distance over full transition matrices.

\subsection{Comparing Invariant Measures}\label{subsec:invariant_comp}

Rather than using the full Markov matrix, one may attempt to extract the optimal parameters $\theta\in\Theta$ by comparing invariant measures. The training objective for invariant measure-based system identification is given by 
\begin{equation}\label{eq:background_IM_loss}
    \mathcal{J}_{\rho}(\theta)
    =
    \mathcal{D}(\rho^{(\varepsilon)}(v_{\theta}),\rho^*),
\end{equation}
where $\mathcal{D}:\mathcal{P}(X)\times \mathcal{P}(X)\to [0,\infty)$ measures the discrepancy between probability measures. In~\eqref{eq:background_IM_loss}, the model-simulated invariant measure $\rho^{(\varepsilon)}(v_{\theta})\in \mathbb{R}^{n}$ is a discrete measure over the cells $\{C_j\}_{j=1}^n$ extracted as the dominant left eigenvector of $\widehat{M}^{(\varepsilon)}(v_{\theta})$; see \eqref{eq:regularized_transition_matrix}. In our experiments, $\rho^*$ is computed as the dominant left eigenvector of the ground truth Markov matrix $M^*$, though any approach for estimating invariant measures can be used. To ensure the existence of a unique dominant eigenvector, we apply teleportation regularization which adds a small nonzero transition probability between all pairs of cells \cite{gleich2015pagerank}. 

Invariant measure matching uses only partial information of the full Markov matrix. In settings where the Markov matrix can be reliably estimated the objective \eqref{eq:background_M_loss} is likely to lead to improved system identification over \eqref{eq:background_IM_loss}. However, in certain applications where data are slowly sampled and noisy, estimation of the Markov matrix can be compromised. In such settings, the ground truth invariant measure $\rho^*$ may still be estimated reliably (see~\cite[Figure~1]{botvinick2023learning}), which motivates the use of \eqref{eq:background_IM_loss}.

\section{Theory}\label{subsec:galerkin}

This section is dedicated to analyzing theoretical properties of our finite-dimensional approximation to the PFO developed in Section \ref{sec:mesh}. In Section \ref{subsec:op_conv}, we consider the noise-free infinite-data limit and study convergence properties of \eqref{eq:regularized_transition_matrix}. In particular, we show strong convergence in an iterated limit where the regularization and cell discretizations are sequentially refined; see Theorem \ref{thm:convergence}. In Section \ref{subsec:MC_analysis}, we then turn to the finite-sample case and analyze the approximation error in constructing the transition matrix \eqref{eq:empirical_markov_background}. This leads us to  derive a principle of optimal variance reduction that informs our choice of cell-placement; see Proposition~\ref{prop:optimal_strat}.

\subsection{Convergence of the Operator}\label{subsec:op_conv}

We now begin our technical discussion.  Let $X$ be a compact metric space and let $\mu\in\mathcal{P}(X)$ be fixed. As in \Cref{def:PFO}, we consider the PFO $\mathcal{T}:L_{\mu}^1(X)\to L_{\mu}^1(X)$ based upon the non-singular discrete map $T:X\to X$. In the continuous-time case, $T$ can simply be taken as the time-$\Delta t$ flow map of some vector field. Now consider a sequence of partitions  $\{\{C_{i,n}\}_{i=1}^n\}_{n=1}^{\infty}$ of $X$ and, for $\varepsilon > 0$, a sequence of partitions of unity $\{\{\psi^{(\varepsilon)}_{i,n}\}_{i=1}^n\}_{n=1}^{\infty}$ belonging to $L_{\mu}^1(X)\cap L_{\mu}^{\infty}(X)$ satisfying \Cref{assumption1} and \Cref{assumption:2}, respectively.
\begin{assumption}\label{assumption1}
The sequence $\{\{C_{i,n}\}_{i=1}^n\}_{n=1}^{\infty}$ satisfies the following. 
    \begin{itemize}
\item \textit{(Partition)}: $C_{i,n}\cap C_{j,n} = \emptyset$ if $i \neq j$ and $\bigcup_{i=1}^n C_{i,n} = X.$
    \item \textit{(Increasing resolution)}: $\displaystyle\lim_{n\to\infty}\max_{1\leq i\leq n}(\textup{diam}(C_{i,n})) = 0$.
    \item \textit{(Positive measure)}: $\mu(C_{i,n}) > 0,$ for all $1\leq i \leq n$ and all $n\in\mathbb{N}$.
\end{itemize}
\end{assumption}
\begin{assumption}\label{assumption:2}
    The collection $\{\{\psi^{(\varepsilon)}_{i,n}\}_{i=1}^n\}_{n=1}^{\infty}$ satisfies the following properties.
    \begin{itemize}
        \item (\textit{Partition of unity}): For $\varepsilon > 0$ and $n\in\mathbb{N}$ fixed,
        $\sum_{i=1}^n \psi_{i,n}^{(\varepsilon)}(x) = 1$ and $\psi_{i,n}^{(\varepsilon)}(x)\geq 0.$ 
        \item (\textit{Approximation}): For $n\in\mathbb{N}$ and $i\leq n$ it holds $\psi_{i,n}^{(\varepsilon)} \xrightarrow[]{\varepsilon \to 0}\chi_{C_{i,n}}$ pointwise $\mu$-a.e.
    \end{itemize}
\end{assumption}
Note that $n\in\mathbb{N}$ represents the resolution of a given partition $\{C_{i,n}\}_{i=1}^n$, which for notational convenience we will occasionally suppress. Now, we define 
\begin{equation}\label{eq:phi_i}
    \phi_{i}:= \frac{\chi_{C_{i}}}{\mu(C_{i})}\in L_{\mu}^1(X) \cap L_{\mu}^{\infty}(X),\qquad 1\leq i \leq n,
\end{equation}
which is a normalized characteristic function on the cell $C_{i}.$ It also holds that $\phi_{i} d\mu = d\mu_{i},$ where $\mu_{i}\in\mathcal{P}(X)$ is the conditional measure defined by $\mu_i(B) :=\mu(C_i\cap B)/\mu(C_i)$, for all $B\in\mathscr{B}$. The subspace of $L_{\mu}^1(X)$ spanned by the collection of functions defined in \eqref{eq:phi_i} will be denoted  $\Delta_n:=\text{span}(\{\phi_{i,n}\}_{i=1}^n) \subseteq L_{\mu}^1(X)$. It is helpful to define the operator $\mathcal{Q}^{(n)}:L_{\mu}^1(X)\to \Delta_n$, where  
\begin{equation}\label{eq:projQ}
\mathcal{Q}^{(n)} f := \sum_{i=1}^n \frac{\langle f, \chi_{C_{i,n}}\rangle_\mu}{\langle \chi_{C_{i,n}},\chi_{C_{i,n}}\rangle_\mu} \chi_{C_{i,n}},
\end{equation}
where $\langle f_1,f_2\rangle_{\mu}:=\int_X f_1 f_2 d\mu$. Note that \eqref{eq:projQ} clearly indicates that $\mathcal{Q}^{(n)} f$ is a projection of $f$ onto $\Delta_n$.
Intuitively, $\mathcal{Q}^{(n)}f$ is a piecewise constant function describing the average value of $f$ over each cell.
Now, for $n\in\mathbb{N}$ and $\varepsilon > 0$, we define $\mathcal{L}^{(n,\varepsilon)}:\Delta_n\to \Delta_n$ by setting 
\begin{equation}\label{eq:ulam_projection_eps}
    \mathcal{L}^{(n,\varepsilon)}\phi_{i,n} := \sum_{j=1}^n M_{i,j}^{(n,\varepsilon)} \phi_{j,n}, \qquad M_{i,j}^{(n,\varepsilon)}:= \frac{\langle\psi_{j,n}^{(\varepsilon)}\circ T \, ,\, \chi_{C_{i},n}\rangle_\mu}{\langle\chi_{C_{i},n}\, , \,\chi_{C_{i},n} \rangle_\mu }.
\end{equation}
Note that $M^{(n,\varepsilon)}$ is the infinite-data limit of the regularized transition matrix defined in \eqref{eq:regularized_transition_matrix}.
In the case when $\psi_{j,n}^{(\varepsilon)} = \chi_{C_{j,n}}$, we have that 
\begin{equation}\label{eq:cond_int}
    M_{i,j} = \int_X \chi_{T^{-1}(C_j)} \phi_{i}d\mu =  \int_X \chi_{T^{-1}(C_j)}d\mu_i,
\end{equation}
and thus \eqref{eq:ulam_projection_eps} reduces to the usual Ulam method in which the entry $M_{i,j}$ describes the transition probability between cell $C_{i}$ and cell $C_{j}$, under the dynamics $T$; see~\eqref{eq:ulam_background}. 

Using the operator $\mathcal{L}^{(n,\varepsilon)}:\Delta_n\to \Delta_n$, we finally define $\mathcal{T}^{(n,\varepsilon)}:L_{\mu}^1(X)\to L_{\mu}^1(X)$ by setting $\mathcal{T}^{(n,\varepsilon)}:= \mathcal{L}^{(n,\varepsilon)}\mathcal{Q}^{(n)},$ which we regard as our approximation to the true PFO. \Cref{thm:convergence} establishes the main approximation properties of $\mathcal{T}^{(n,\varepsilon)}$. In particular, for any finite $n\in\mathbb{N}$ and $\varepsilon > 0$ we show that the operator preserves positivity and satisfies the mass-conservation principle. Moreover, as the parameters $n\in\mathbb{N}$ and $\varepsilon > 0$ are refined, we obtain strong operator convergence to the true PFO.
\begin{theorem}\label{thm:convergence}
    For $\varepsilon > 0$ and $n\in\mathbb{N}$ fixed,  the operator $\mathcal{T}^{(n,\varepsilon)}:L_{\mu}^1(X)\to L_{\mu}^1(X)$ is Markov. Moreover, we obtain the following strong operator convergence:
    $$\lim_{n\to \infty}\lim_{\varepsilon \to 0} \|\mathcal{T}^{(n,\varepsilon)}f - \mathcal{T}f\|_{L_{\mu}^1} = 0,\qquad \forall f\in L_{\mu}^1(X).$$
\end{theorem}
Our proof of \Cref{thm:convergence} is presented in Section \ref{subsec:3_conv}. 

\subsubsection{Proof of Convergence.}\label{subsec:3_conv}
We begin by establishing several lemmas which are necessary for our proof of \Cref{thm:convergence}. Many of the following lemmas generalize results from \cite{li1976finite} from the setting of one-dimensional interval maps to arbitrary compact metric spaces. In what follows, we omit the superscript $\varepsilon$ from the operators introduced in Section~\ref{sec:mesh} when considering the case $\psi_{j,n}^{(\varepsilon)} = \chi_{C_j}. $  We begin by checking that the projected PFO $\mathcal{L}^{(n)}$, which is clearly linear, is also positive and mass-preserving.
\begin{lemma}\label{lemma:g1}
If $h\in \Delta_n$ satisfies $h\geq 0$, then $\mathcal{L}^{(n)}h \geq 0$ as well and $\int_X \mathcal{L}^{(n)}h d\mu = \int_X h d\mu.$
\end{lemma}
\begin{proof}
 Write $h = \sum_{i=1}^n a_i \phi_{i,n}$, with $a_i \geq 0$ for $1\leq i \leq n$. Then, by linearity we have that 
\begin{equation}\label{eq:eq1}
    \mathcal{L}^{(n)}h  = \sum_{i=1}^n a_i \mathcal{L}^{(n)}\phi_{i,n} = \sum_{i=1}^n \sum_{j=1}^n a_i M_{i,j}^{(n)}\phi_{j,n}.
\end{equation}
Since $a_i,M_{i,j}^{(n)}\geq 0$ for each $1\leq i,j\leq n$, it holds that $\mathcal{L}^{(n)}h \geq 0$. Moreover, using the above expression for $\mathcal{L}^{(n)}h$, we have
\begin{align*}
    \int_X \mathcal{L}^{(n)}h d\mu =\sum_{i=1}^n \sum_{j=1}^n\bigg( a_i M_{i,j}^{(n)}\int_X\phi_{j,n}d\mu\bigg) = \sum_{i=1}^n a_i \sum_{j=1}^n M_{i,j}^{(n)} = \sum_{i=1}^n a_i = \int_X h d\mu,
\end{align*}
which completes the proof.
\end{proof}
We next establish that the projection $\mathcal{Q}^{(n)}$ is also positive and mass-preserving. 
\begin{lemma}\label{lemma:g2}
If $f\in L_{\mu}^1(X)$ satisfies $f\geq 0$, then $\mathcal{Q}^{(n)} f \geq 0$ and $\int_X \mathcal{Q}^{(n)} f d\mu = \int_X fd\mu$. Moreover, for any $g\in L_{\mu}^1(X)$ we have $\int_X |\mathcal{Q}^{(n)} g|d\mu \leq \int_X |g|d\mu$.
\end{lemma}
\begin{proof}
Assume that $f\in L_{\mu}^1(X) $ satisfies $f\geq 0.$ Then the positivity of $\mathcal{Q}^{(n)} f$ is clear from the definition given in~\eqref{eq:projQ}, as $\int_{C_{i,n}} f d\mu \geq 0$, for each $1\leq i \leq n$. We now verify the mass-conservation property. Indeed, notice that
    \begin{align*}
    \int_X \mathcal{Q}^{(n)} f d\mu &=\int_{X}\sum_{j=1}^n \bigg(\frac{1}{\mu(C_{j,n})} \int_{C_{j,n}} f(y) d\mu(y)\bigg)\chi_{C_{j,n}} d\mu(x) \\
    & = \sum_{j=1}^n \int_{C_{j,n}}f(y)d\mu(y) = \int_X f d\mu.
    \end{align*}

   We now prove the last part of the lemma, stating that $\mathcal{Q}^{(n)}$ is a contraction. Indeed, when $g$ has both positive and negative parts, we can decompose $g = g^+ - g^-$ where $g^+,g^-\geq 0$, $|g| = g^+ + g^-$, and apply linearity of $\mathcal{Q}^{(n)}$ to deduce 
    $$\int_X |\mathcal{Q}^{(n)} g| d\mu \leq \int_X \mathcal{Q}^{(n)} g^+d\mu + \int_X \mathcal{Q}^{(n)} g^- d\mu = \int_X g^+ d\mu + \int_X g^- d\mu = \int_X |g|d\mu.$$
    Above, we have also used the fact that $\mathcal{Q}^{(n)}$ is positive.
\end{proof}

The following lemma highlights the relationship between $\mathcal{T}^{(n)}$ and $\mathcal{T}$.
\begin{lemma}\label{lemma:g3}
    It holds that $\mathcal{T}^{(n)} =  \mathcal{Q}^{(n)}\mathcal{T}\mathcal{Q}^{(n)} .$
\end{lemma}
\begin{proof}
    It suffices to show that 
    \begin{equation}\label{eq:wts}
    \mathcal{L}^{(n)} h = \mathcal{Q}^{(n)}\mathcal{T} h,\qquad \text{where\,\,}h = \sum_{i=1}^n a_i \phi_{i,n}.
    \end{equation}
    Towards this, it follows by \eqref{eq:projQ} that 
    $$
        \mathcal{Q}^{(n)} \mathcal{T} h = \sum_{j=1}^n \sum_{i=1}^n \bigg( a_i \int_{C_{j,n}} \mathcal{T}\phi_{i,n} d\mu\bigg)  \phi_{j,n}= \sum_{j=1}^n \sum_{i=1}^n a_i M^{(n)}_{i,j}\phi_{j,n}= \mathcal{L}^{(n)}h,
   $$
    where the final equality is a result of \eqref{eq:eq1}.
\end{proof}

We now rely on the fact that $X$ is a compact metric space, writing $d(\cdot,\cdot)$ to denote the metric on $X$, in order to prove the convergence of  $\mathcal{Q}^{(n)}$; see \Cref{lemma:g4}
\begin{lemma}\label{lemma:g4}
    For all $f\in L_{\mu}^1$, it holds that $\mathcal{Q}^{(n)}f \to f$ in $L_{\mu}^1(X)$. 
\end{lemma}
\begin{proof}
    Note that for any $\varepsilon > 0$, we can find some $g\in C(X)$ such that $\|g-f\|_{L_{\mu}^1} < \varepsilon/3$. Since $X$ is compact $g$ is uniformly continuous.  Thus, there exists $\delta > 0$ such that if $d(x,y)<\delta$, then $|g(x)-g(y)| < \varepsilon/3$. By the assumption that $\displaystyle\lim_{n\to\infty}\max_{1\leq i\leq n}(\textup{diam}(C_{i,n})) = 0$, we may choose $N\in\mathbb{N}$ such that for all $n\geq N$ it holds that $\textup{diam}(C_{i,n}) < \delta$, for each $1\leq i \leq n$. Thus, for all $n \geq N$ if $x,y\in C_{i,n}$, for some $1\leq i \leq n$, it necessarily holds that $|g(x)-g(y)|<\varepsilon/3$. Using this fact, we have for $n\geq N$ and $1\leq i \leq n$ that
    \begin{align*}
        \quad \int_{C_{i,n}}  |\mathcal{Q}^{(n)} g(x) - g(x)|d\mu(x) 
        &=\int_{C_{i,n}}\bigg| \sum_{j=1}^n \bigg(\frac{1}{\mu(C_{j,n})} \int_{C_{j,n}} g(y) d\mu(y) \bigg) \chi_{C_{j,n}}(x) - g(x)\bigg| d\mu(x)\\
        &=\int_{C_{i,n}}\bigg| \frac{1}{\mu(C_{i,n})} \int_{C_{i,n}} g(y) d\mu(y)  - g(x)\bigg| d\mu(x)\\
        &\leq \int_{C_{i,n}}\frac{1}{\mu(C_{i,n})} \int_{C_{i,n}}| g(y)-g(x)| d\mu(y)  d\mu(x)\leq \frac{\varepsilon\mu(C_{i,n}) }{3}.
    \end{align*}
    It then follows that 
    $$\int_X |\mathcal{Q}^{(n)} g - g| d\mu = \sum_{i=1}^n \int_{C_{i,n}} |\mathcal{Q}^{(n)}g -g| d\mu< \frac{\varepsilon}{3} \sum_{i=1}^n \mu(C_{i,n}) = \frac{\varepsilon}{3}. $$
   Putting everything together, we now have 
    \begin{align*}
        \int_X|\mathcal{Q}^{(n)} f - f| d\mu &\leq \int_X |\mathcal{Q}^{(n)} f - \mathcal{Q}^{(n)} g| d\mu + \int_X |\mathcal{Q}^{(n)} g - g |d\mu + \int_X |f - g|d\mu \\
        &\leq 2 \int_X |f-g|d\mu + \int_X |\mathcal{Q}^{(n)} g - g|d\mu \leq \varepsilon,
    \end{align*}
    which concludes the proof.
\end{proof}
We now prove \Cref{thm:g1}, which relies on Lemmas \ref{lemma:g1}, \ref{lemma:g2}, \ref{lemma:g3}, and \ref{lemma:g4}. \Cref{thm:g1} should be viewed as the version of \Cref{thm:convergence} without any regularization.
\begin{theorem}\label{thm:g1}
    $\mathcal{T}^{(n)}$ is linear, positive, and Markov. Moreover, $$\lim_{n\to \infty}\|\mathcal{T}^{(n)}f- \mathcal{T}f\|_{L_{\mu}^1} = 0, \qquad \forall f\in L_{\mu}^1(X).$$
\end{theorem}
\begin{proof}[Proof of \Cref{thm:g1}]
The fact that $\mathcal{T}^{(n)}$ is a Markov operator is a direct consequence of \Cref{lemma:g1} and \Cref{lemma:g2}, as the composition of mass-preserving operators is still mass-preserving, and the composition of positive operators is still positive. We now prove the convergence result. Let $f\in L_{\mu}^1(X)$ be fixed and notice that 
    \begin{align}
        \| (\mathcal{T}^{(n)} - \mathcal{T}) f \|_{L_{\mu}^1} &= \| \mathcal{Q}^{(n)} \mathcal{T} \mathcal{Q}^{(n)}f- \mathcal{T} f \|_{L_{\mu}^1} \label{eq:exp1}\\
        & \leq \| \mathcal{Q}^{(n)} \mathcal{T} \mathcal{Q}^{(n)}f - \mathcal{Q}^{(n)} \mathcal{T} f\|_{L_{\mu}^1} + \|\mathcal{Q}^{(n)} \mathcal{T} f- \mathcal{T} f \|_{L_{\mu}^1} \label{eq:exp3} \\
        & \leq \|\mathcal{T} \mathcal{Q}^{(n)} f  - \mathcal{T} f\|_{L_{\mu}^1} + \|\mathcal{Q}^{(n)} \mathcal{T} f- \mathcal{T} f \|_{L_{\mu}^1}\xrightarrow[]{n\to\infty} 0.  \label{eq:exp2} 
    \end{align}
    Above, \eqref{eq:exp1} follows from \Cref{lemma:g3}. Moreover, \eqref{eq:exp2} follows from \eqref{eq:exp3}, as it was established that $\mathcal{Q}^{(n)}$ is a contraction on $L_{\mu}^1(X)$ in \Cref{lemma:g2}. The first term in \eqref{eq:exp2} then goes to zero as $\mathcal{T}$ is a bounded linear operator, due to the fact that it is Markov \cite[Proposition 3.1.1]{lasota2013chaos}, and hence continuous. Thus, since $\mathcal{Q}^{(n)} f \to f$ in $L_{\mu}^1(X)$ by \Cref{lemma:g4}, it holds that $\mathcal{T}\mathcal{Q}^{(n)}f \to \mathcal{T}f$ in $L_{\mu}^1(X)$, as well. The second term in \eqref{eq:exp2} also goes to zero as a consequence of \Cref{lemma:g4}. 
\end{proof}
We can now prove \Cref{thm:convergence}, which establishes convergence of the regularized Galerkin projection appearing in Section \ref{sec:mesh}.
\begin{proof}[Proof of \Cref{thm:convergence}]
    \Cref{lemma:g2} shows that $\mathcal{Q}^{(n)}:L_{\mu}^1(X)\to \Delta_n$ is mass-preserving and positive. Thus, to show that $\mathcal{T}^{(n,\varepsilon)} = \mathcal{L}^{(n,\varepsilon)}\mathcal{Q}^{(n)}$ is Markov, it suffices to show that $\mathcal{L}^{(n,\varepsilon)}:\Delta_n\to \Delta_n$ is also mass-preserving and positive. We remark that positivity is clear from the definition in \eqref{eq:ulam_projection_eps}, as $\phi_{i,n}(x), \psi_{i,n}^{(\varepsilon)}(x)\geq 0$ for all $x\in X$,  $n\in\mathbb{N}$, $i\leq n$, and $\varepsilon > 0$. Thus, we will focus on verifying the mass-preservation property. Towards this, let $h\in \Delta_n$ be fixed with $h\geq 0$ and write
$h = \sum_{i=1}^n a_i \phi_{i,n}$. Now, notice that 
\begin{align}
    \int_X \mathcal{L}^{(n,\varepsilon)}hd\mu = \sum_{i=1}^n a_i \bigg(\int_X \mathcal{L}^{(n,\varepsilon)} \phi_{i,n} d\mu\bigg)
    &= \sum_{i=1}^n a_i \sum_{j=1}^n\bigg( \int_X  M_{i,j}^{(n,\varepsilon)}\phi_{j,n}d\mu\bigg) \label{eq:l1}\\
    &= \sum_{i=1}^n a_i \sum_{j=1}^nM_{i,j}^{(n,\varepsilon)}\label{eq:l2}\\
    &=\sum_{i=1}^n a_i \bigg(\int_X \sum_{j=1}^n\psi^{(\varepsilon)}_{j,n}\circ T\phi_{i,n}d\mu\bigg) \nonumber\\
    &= \int_X h d\mu.\label{eq:l3}
\end{align}
Above, \eqref{eq:l1} follows from linearity, \eqref{eq:l2} uses the fact that $\phi_{j,n}$ has unit integral, and \eqref{eq:l3} leverages the fact that $\{\psi_{j,n}^{(\varepsilon)}\}_{j=1}^n$ forms a partition of unity. The rest of the equalities follow from simple rearrangements. Thus, we have established that $\mathcal{T}^{(n,\varepsilon)}$ is Markov.

We now focus on proving the convergence. Let $n\in\mathbb{N}$ be fixed and assume $i,j\leq n$. Since $\psi_{j,n}^{(\varepsilon)} \xrightarrow[]{\varepsilon \to 0} \chi_{C_{j,n}}$ pointwise $\mu$-almost everywhere and $T$ is non-singular, we have that 
\begin{equation}\label{eq:construction1}
    g_{\varepsilon}(x):= \Big| (\psi_{j,n}^{(\varepsilon)}\circ T)(x) - (\chi_{C_{j,n}}\circ T)(x) \Big| \phi_{i,n}(x)
\end{equation}
satisfies $g_{\varepsilon}(x) \xrightarrow[]{\varepsilon \to 0} 0$ pointwise $\mu$-almost everywhere. Moreover, by the construction in \eqref{eq:construction1} we have $|g_{\varepsilon}(x)| \leq \phi_{i,n}(x)$ for all $\varepsilon > 0$ and all $x\in X$. Thus, it follows by Lebesgue's dominated convergence theorem that
\begin{align*}
    |M_{i,j}^{(n,\varepsilon)} - M_{i,j}^{(n)}| \leq \int_X \Big| \psi_{j,n}^{(\varepsilon)}\circ T - \chi_{C_{j,n}}\circ T \Big| \phi_{i,n}d\mu=\int_X g_{\varepsilon}d\mu \xrightarrow[]{\varepsilon \to 0} 0.
\end{align*}
Using this result, notice that for all $x\in X$ it holds that
$$\lim_{\varepsilon \to 0}(\mathcal{L}^{(n,\varepsilon)} \phi_{i,n})(x) = \lim_{\varepsilon \to 0 } \sum_{j=1}^n M_{i,j}^{(n,\varepsilon)}\phi_{j,n}(x) = \sum_{j=1}^n M_{i,j}^{(n)}\phi_{j,n}(x) = \mathcal{L}^{(n)} \phi_{i,n}(x).$$
Moreover, since $$|\mathcal{L}^{(n,\varepsilon)}\phi_{i,n}| \leq \sum_{j=1}^n M_{i,j}^{(n,\varepsilon)} \phi_{j,n} \leq \sum_{j=1}^n \phi_{j,n}\in L_{\mu}^1(X)$$ the dominated convergence theorem allows us to convert the pointwise convergence above into $L_{\mu}^1$ convergence. That is,
$$\lim_{\varepsilon \to 0}\|\mathcal{L}^{(n,\varepsilon)}\phi_{i,n} - \mathcal{L}^{(n)}\phi_{i,n}\|_{L_{\mu}^1} = 0.$$
The result then extends for arbitrary $h\in \Delta_n$. That is, if we write $h = \sum_{i=1}^na_i \phi_{i,n}$ we have by the triangle inequality that 
\begin{equation}\label{eq:eps_conv}
    \lim_{\varepsilon \to 0}\|\mathcal{L}^{(n,\varepsilon)}h - \mathcal{L}^{(n)}h\|_{L_{\mu}^1} \leq \lim_{\varepsilon \to 0} \sum_{i=1}^n |a_i| \|\mathcal{L}^{(n,\varepsilon)}\phi_{i,n} - \mathcal{L}^{(n)}\phi_{i,n}\|_{L_{\mu}^1} = 0.
\end{equation}
Finally, for any arbitrary $f\in L_{\mu}^1(X)$ we have that 
\begin{align*}
    \lim_{n\to \infty}\lim_{\varepsilon \to 0}\| \mathcal{T}^{(n,\varepsilon)} f - \mathcal{T} f\|_{L_{\mu}^1}&= \lim_{n\to \infty}\lim_{\varepsilon \to 0}\| \mathcal{L}^{(n,\varepsilon)}\mathcal{Q}^{(n)}f - \mathcal{T} f\|_{L_{\mu}^1} \nonumber \\
    &\leq \lim_{n\to \infty}\lim_{\varepsilon \to 0}\| \mathcal{L}^{(n,\varepsilon)}\mathcal{Q}^{(n)}f  -\mathcal{L}^{(n)}\mathcal{Q}^{(n)}f\|_{L_{\mu}^1}+ \lim_{n\to \infty}\lim_{\varepsilon \to 0}\| \mathcal{T}^{(n)}f - \mathcal{T} f\|_{L_{\mu}^1} \nonumber\\
    & = \lim_{n\to \infty} \|\mathcal{T}^{(n)}f - \mathcal{T}f\|_{L_{\mu}^1}= 0.
\end{align*}
Above, we have used the convergence in \eqref{eq:eps_conv} to move from the second line to the third line, as well as the convergence result from \Cref{thm:g1} to obtain the final equality. 
\end{proof}

\subsection{Optimal Partition Construction}\label{subsec:MC_analysis}
While in Section \ref{subsec:galerkin}, we established the convergence of a PFO approximation scheme in the infinite-data and infinite-resolution limit (see Theorem \ref{thm:convergence}), in practice we must always perform computations using a finite amount of data and a finite number of cells. In this section, we investigate how the construction of the partition $\{C_i\}_{i=1}^n$ impacts our approximation error. While the optimal choice of partition for Ulam's method has previously been studied in the literature in the infinite-data limit, see e.g.~\cite{murray2004optimal}, we focus  on characterizing optimality from the perspective of finite-sample variance reduction. We formalize our optimality principle as a minimax problem and we show that  constructing a partition which \textit{evenly distributes} the observed data across the cells satisfies this optimality principle (see Proposition \ref{prop:optimal_strat}). In practice, this can be approximately achieved using a $k$-means clustering of the observed dataset, and in certain situations a constrained $k$-means routine can build clusters which exactly satisfy this condition \cite{bradley2000constrained}.

For simplicity, we consider the case without regularization, i.e., we assume $\psi_j^{(\varepsilon)} = \chi_{C_j}$ for $1\leq j \leq n$. We assume access to $N$ i.i.d.~samples $\{x_k\}_{k=1}^{N}\sim \mu\in\mathcal{P}(X)$, as well as the pairings $\{(x_k,T(x_k))\}_{k=1}^N$ under some dynamical system $T:X\to X$. That is, we do not know the evolution rule $T$ a priori, but we instead observe $N$~i.i.d. samples describing its action on a fixed data-distribution. We will also assume throughout that $\text{mod}(N,n)  = 0$, i.e., it is possible to split the observed samples into $n$ subsets of size $N/n$. 

Let us now write $N_i:=|\{x_k\in C_i\}|$ to denote the number of samples in the $i$-th cell, which we assume to be positive. We also define $\{x_k^i\}_{k=1}^{N_i}$ as the subset of $\{x_k\}_{k=1}^N$ which is contained in the cell $C_i$.
Since the samples $\{x_k\}_{k=1}^N$ are i.i.d.~with respect to $\mu$, it also holds that the collection $\{x_k^i\}_{k=1}^{N_i}$ of samples which are contained in $C_i$ are i.i.d.~with respect to $\mu_i$; see Section \ref{subsec:galerkin}. In practice, one represents the PFO using the matrix $M\in\mathbb{R}^{n\times n}$ defined in~\eqref{eq:cond_int}, and with access to only finitely many samples from $\mu$, the entries of $M$ are approximated via Monte Carlo integration or ergodic averages; see \eqref{eq:empirical_markov_background}.
The approximation in \eqref{eq:empirical_markov_background} has known standard deviation
\begin{equation}\label{eq:err}
D_{i,j}:=\Big(\mathbb{E}[(M_{i,j}-\widehat{M}_{i,j})^2]\Big)^{1/2} = \sqrt{\frac{\text{Var}_{\mu_i}(\chi_{T^{-1}(C_j)})}{N_i}}, \qquad 1\leq i,j\leq n ;
\end{equation}
see \cite[Section 2]{Caflisch_1998}. We now state our principle of optimal variance reduction, which seeks a partition of $X$ minimizing the worst-case standard deviation.
\begin{variance1}
    Let $\mathcal{X}_n$ denote the set of all  partitions of $X$ into $n$ subsets. Then, $\{\hat{C}_{i}\}\in \mathcal{X}_n$ satisfies the principle of optimal variance reduction (V1) if 
    \begin{equation}\label{eq:variance_1}
        \{\hat{C}_i\} \in \argmin_{\{C_i\} \in \mathcal{X}_n } \max_{1\leq i,j \leq n} D_{i,j},
    \end{equation}
    where $D_{i,j}$ is the standard deviation defined in \eqref{eq:err} based on a partition $\{C_i\} \in \mathcal{X}_n$
\end{variance1}
Without prior knowledge of the underlying system $T$ or the measure $\mu$, which are both required to evaluate the numerator of \eqref{eq:err}, we cannot check if a given partition $\{C_i\}\in \mathcal{X}_n$ satisfies our principle of optimal variance reduction. Notice that the numerator of \eqref{eq:err} is given by 
$$\text{Var}_{\mu_i}(\chi_{T^{-1}(C_j)}) =  \int_X \chi_{T^{-1}(C_j)} d\mu_i  - \bigg(\int_X \chi_{T^{-1}(C_j)}d\mu_i\bigg)^2 = M_{i,j}-M_{i,j}^2.$$
Since $M_{i,j}-M_{i,j}^2\in [0,1/4]$, we instead consider a modification to our principle of optimal variance reduction which places a worst-case bound on the standard deviation in \eqref{eq:err} and can be verified when only finite-sample data is available.
\begin{variance2}
A partition $\{\hat{C}_i\}\in \mathcal{X}_n$ satisfies the principle of optimal variance reduction (V2) if 
\begin{equation}\label{eq:solve_arg}
 \{\hat{C}_i\}\in \argmin_{\{C_i\}\in \mathcal{X}_n} \max_{1\leq i \leq n}\frac{1}{\sqrt{N_i}}.   
\end{equation}
\end{variance2}
Note that the minimax problem introduced in \eqref{eq:solve_arg} depends only on the number of points contained in each cell. Thus, solving \eqref{eq:solve_arg} is equivalent to finding numbers $N_1,\dots,N_n\in\mathbb{N}$ which solve
\begin{equation}\label{eq:minimizesum}
   \argmin_{N_1,\dots,N_n\in\mathbb{N}} \max_{1\leq i \leq n}\frac{1}{\sqrt{N_i}} \qquad \text{such that}\qquad \sum_{i=1}^n N_i = N.
\end{equation}
By inspection, one can verify that the solution to \eqref{eq:minimizesum} is given by $N_i = N/n$ for each $1\leq i \leq n$. This leads us to Proposition \ref{prop:optimal_strat} which classifies the set of partitions satisfying the principle of optimal variance reduction (V2).
\begin{proposition}\label{prop:optimal_strat}
    A partition $\{\hat{C}_i\}\in \mathcal{X}_n$  solves \eqref{eq:solve_arg} if and only if $N_i = N/n$ for all $1\leq i \leq n$. 
\end{proposition}

Proposition \ref{prop:optimal_strat} indicates that we should build partitions which contain the same number of samples in each cell, in order to reduce the worst-case standard deviation when approximating the PFO via the~strategy given in \eqref{eq:empirical_markov_background}. Notably, if the data samples come from a long trajectory which samples the invariant measure, then this mesh-construction strategy will involve focusing the majority of cells on the support of the invariant measure. We therefore expect this scheme to be especially useful for approximating the PFO of high-dimensional systems admitting low-dimensional invariant measures. In Section \ref{sec:numerics}, we implement the selection principle suggested by Proposition \ref{prop:optimal_strat} via $k$-means clustering and demonstrate the improved efficiency of the data-adaptive unstructured mesh compared to a uniform mesh; see Figure \ref{fig:Cat}.

\section{Numerical Results}\label{sec:numerics}
\subsection{Overview}
In this section, we provide comprehensive numerical experiments benchmarking our data-adaptive approaches for learning dynamical systems from invariant measures and Markov matrices\footnote{Our code is publicly available: \url{https://github.com/Yinonghyn/Parameter-Identification-with-Data-Driven-Cells}}. We begin in Section~\ref{MESH} by comparing the uniform and adaptive meshes for approximating the invariant measure of a low-dimensional dynamical system. Our results demonstrate that the data-adaptive approach is superior in terms of both computational cost and accuracy when the invariant measure is concentrated on a smaller region of the state space; see Figure \ref{fig:Cat}. Section \ref{MAT} is then dedicated to benchmarking the performance of the matrix matching objective \eqref{eq:background_M_loss}, while Section \ref{IM} benchmarks the invariant measure-matching objective \eqref{eq:background_IM_loss}.  Finally, in \Cref{SST}, we apply the Markov matrix approach to a real-world sea surface temperature (SST) prediction problem.

\paragraph{The dynamical systems.} The experiments in Sections  \ref{IM} and \ref{MAT} consider the Lorenz-63 system given by
\begin{equation}
\begin{cases}
\dot{x} = \sigma (y - x), \\
\dot{y} = x(\rho - z) - y, \\
\dot{z} = xy - \beta z,
\end{cases}
\label{eq:lorenz63}
\end{equation}
where $\sigma = 10$, $\rho = 28$, and $\beta = 8/3$. These parameter values place the system in the chaotic regime. Our experiments also consider the Lorenz-96 model, which is a standard test problem for data assimilation and forecasting in high-dimensional chaotic systems. For state variables $x_1,\ldots,x_d$, the dynamics are
\begin{equation}
\frac{dx_i}{dt}
=
(x_{i+1} - x_{i-2})x_{i-1} - x_i + F,
\qquad i = 1,\ldots,d,
\label{eq:lorenz96}
\end{equation}
where the indices are interpreted cyclically modulo $d$, so that $x_{i+d}=x_i$ for all $i$.  

\paragraph{The unknown dynamics.} When considering the Markov matrix matching loss for the Lorenz-63 system we reconstruct the vector field $v:\mathbb{R}^3\to \mathbb{R}^3$ on the full state. For the invariant measure loss, which cannot uniquely reconstruct the full dynamics, we restrict to learning the first component $\dot{x}$ of the vector field, while treating $\dot{y}$ and $\dot{z}$ as known. When considering the Markov matrix matching objective for the Lorenz-96 system with $d = 5$ and $F = 8$, we treat all state variables as unknown and learn the vector field $v:\mathbb{R}^5\to\mathbb{R}^5$ on the full state. When $d = 30$, we only learn the first component $\dot{x}_1:\mathbb{R}^{30}\to\mathbb{R}$ of the vector field and assume that all other components $\dot{x}_2,\dots,\dot{x}_{30} $ are fixed and  known a priori. Finally, when considering the invariant measure matching objective for the Lorenz-96 system with $d = 5$, we assume that the forcing $F:\mathbb{R}^5 \to \mathbb{R}$ is an unknown function, while the remaining components of the vector field are all known. For the NOAA sea-surface temperature forecasting experiment presented in Section \ref{SST}, we attempt to learn the reduced-order dynamics using both the Markov matrix matching and invariant measure matching objectives.

\paragraph{Baseline comparisons.} The main approach we compare the training schemes \eqref{eq:background_M_loss} and \eqref{eq:background_IM_loss} with is a pointwise least-squares difference between the simulated and observed flow maps, i.e., the one-step rollout loss 
\begin{equation}\label{eq:least_square}
    \mathcal{J}(\theta) = \frac{1}{N}\sum_{k=0}^{N-1}|x_{k+1} - T_{\theta}(x_k)|^2,
\end{equation}
where $\{x_k\}_{k=0}^{N}$ is the observed data and $x_{k+1}= T^*(x_k)$ where $T^*$ is the ground truth flow map. The loss \eqref{eq:least_square} is common in shooting methods and for training neural ODEs \cite{DBLP:journals/corr/abs-1806-07366} when $v_{\theta}$ is parameterized as a neural network. We also compare our proposed frameworks with SINDy~\cite{brunton2016discovering}, which learns the dynamics by taking divided difference approximations of the observed trajectory and performing sparse regression onto a library of candidate terms.

\paragraph{Model training.} All approaches are trained using the same feedforward neural network architecture with hyperbolic tangent activation and the Adam optimizer with a learning rate of $10^{-3}$ \cite{kingma2014adam}. Unless otherwise stated, all models use three hidden layers with 100 nodes each. Moreover, all invariant measure matching experiments use the Euclidean distance between the discrete probability vectors, while the loss function used for Markov matrix matching is specified throughout.  To ensure fair comparison of the various training objectives, we terminate optimization after the objective function has been reduced below a fixed percentage of its initial value. This  controls for differences in optimization difficulty across the various objectives that arise due to the intrinsic biases of neural network parameterizations. All data is normalized before neural network training. 

\paragraph{Assessing model performance.} After training, we assess model performance using two complementary metrics. The first is a one-observation-step root mean squared error (RMSE) comparison between the unseen ground truth and learned flow maps, capturing how well the learned model can forecast short-term transient behavior. Similar to \eqref{eq:least_square}, when computing the RMSE metric, the model is initialized at multiple ground-truth states and is assessed based on its prediction of the next observation.  The second is the 2-Wasserstein distance between empirical measures generated through long trajectory simulations for the learned and observed dynamics.  This error metric captures the alignment between the model's asymptotic statistical behavior and the underlying invariant measure. While in many examples we train on noisy data, these two error metrics are always evaluated with respect to a clean, unseen testing set. When training data are sparsely sampled and noisy, we consistently find that our Markov matrix and invariant measure matching objectives offer improved robustness over both SINDy and the one-step rollout loss \eqref{eq:least_square}.

\subsection{Uniform and Data-Adaptive Mesh Comparison}\label{MESH}

We first compare uniform and data-adaptive unstructured meshes for approximating invariant densities. The purpose of this experiment is to isolate the effect of the state-space partition before using the resulting Markov matrices as objectives for learning dynamics. A uniform mesh distributes cells evenly over the entire domain, whereas a data-adaptive mesh concentrates cells in regions visited frequently by the observed trajectory. This distinction is important for systems whose invariant measure is strongly nonuniform.

We consider a two-dimensional discrete-time system on $[0,1]^2$. Let $S:[0,1]^2\to[0,1]^2$ be defined by
\begin{equation}\label{eq:cat}
    S = g\circ C \circ g^{-1},
    \qquad
    g(x,y)=(x^{1/10},y),
    \qquad
    C(x,y)=(2x+y,x+y)\pmod{1}.
\end{equation}
Here $C$ is the Arnold cat map, which preserves Lebesgue measure on $[0,1]^2$ and is a standard example of a chaotic area-preserving map~\cite{lasota2013chaos}. The nonlinear change of variables $g$ transforms this system into one with a nonuniform invariant density, denoted by $\rho^*$. This example therefore provides a simple test case in which an adaptive mesh should be advantageous: most samples concentrate in a smaller portion of the domain, while a uniform mesh continues to allocate cells to regions with little invariant mass.

\begin{figure}[h!]
    \centering
    \subfloat[Ground truth invariant density and density approximations obtained using a uniform mesh and a data-adaptive unstructured mesh with $400$ cells.]{
        \label{fig:CatA}
        \includegraphics[width=0.95\textwidth]{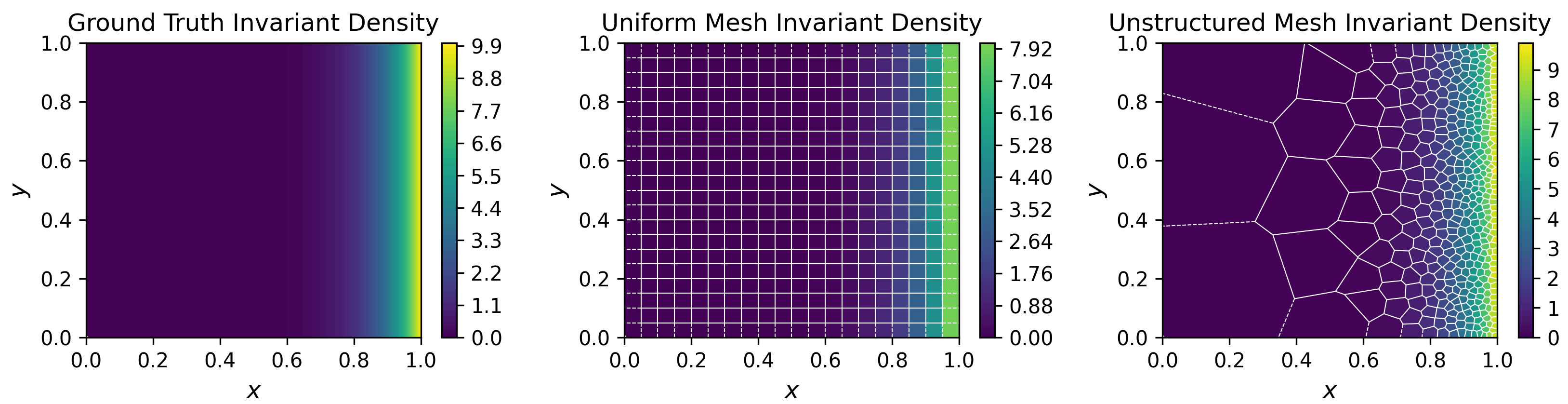}
    }
    \\
    \vspace{.5cm}
    \subfloat[Convergence comparison between uniform and data-adaptive meshes in the large-data regime.]{
        \label{fig:CatB}
        \includegraphics[width=.35\textwidth]{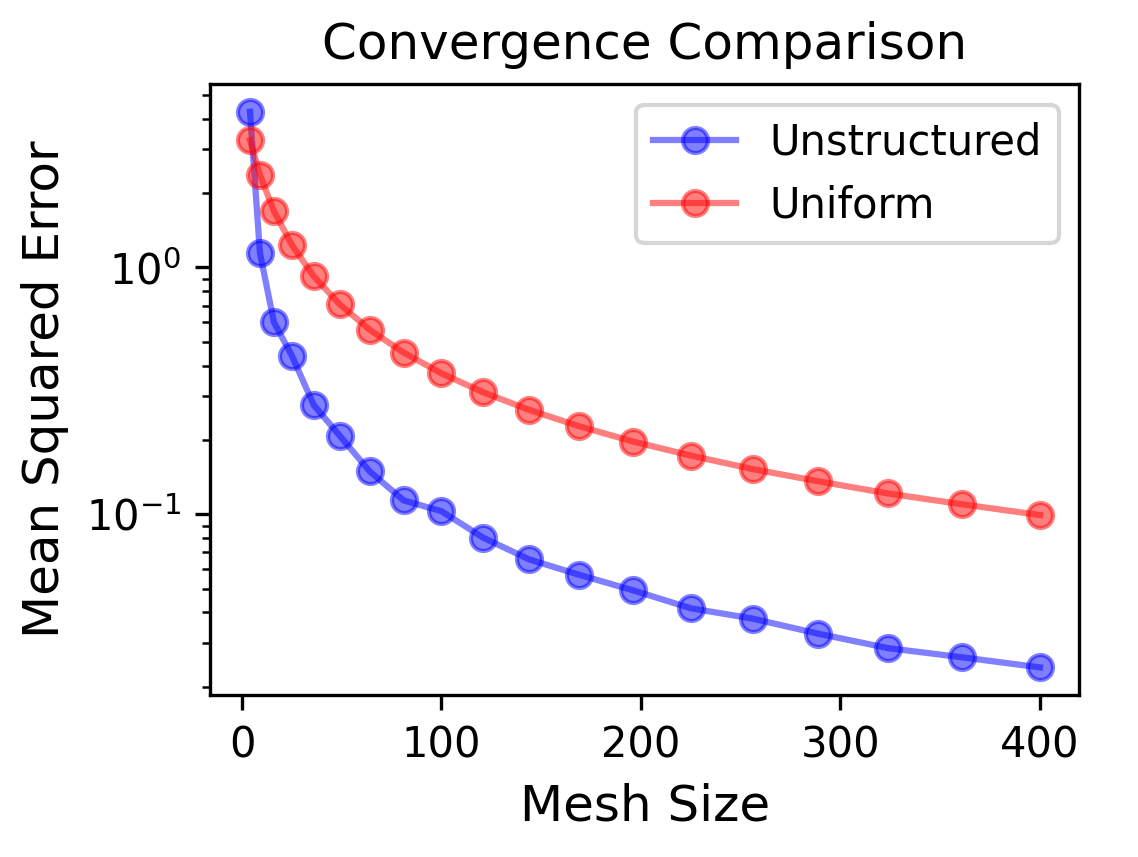}
    }
    \qquad
    \subfloat[Efficiency comparison between the two mesh constructions for a fixed number of cells.]{
        \label{fig:CatC}
        \includegraphics[width=.36\textwidth]{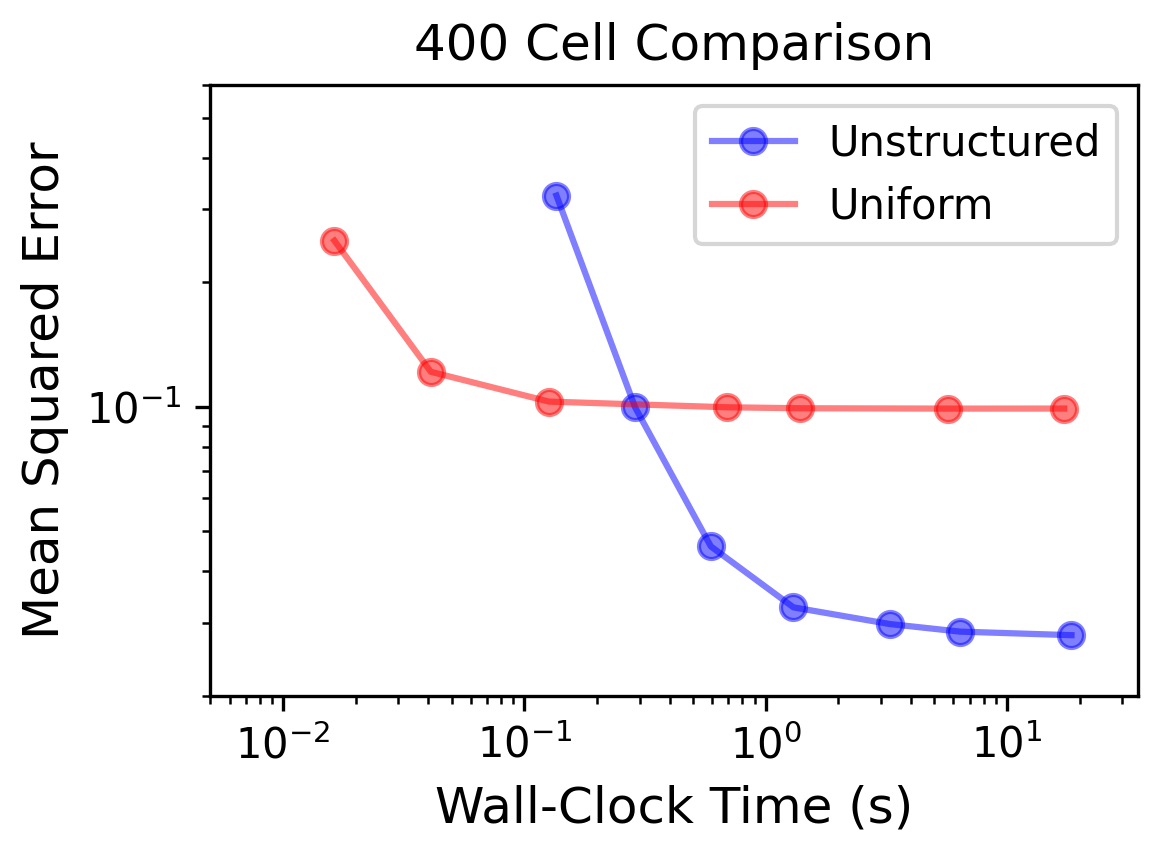}
    }
    \caption{Comparison between uniform and data-adaptive unstructured meshes for approximating the invariant density of \eqref{eq:cat}. In panels (a) and (b), we use $N=10^4$ initial conditions and $K=10^3$ iterations of the map. In panel (c), we fix the number of cells at $n=400$ and vary the number of iterations $K$ to compare accuracy and computational cost.}
    \label{fig:Cat}
\end{figure}

We generate data of the form $
    \big\{ \{S^\ell(x_k)\}_{k=1}^N \big\}_{\ell=1}^K,$
where the initial conditions $\{x_k\}_{k=1}^N$ are sampled independently from the uniform distribution on $[0,1]^2$. Here $N$ denotes the number of initial conditions and $K$ denotes the number of iterations applied to each initial condition. From these trajectory samples, we construct the empirical Markov matrix following \eqref{eq:empirical_markov_background} and compute the approximate invariant measure from its dominant eigenvector.

\Cref{fig:Cat} compares the uniform and data-adaptive mesh approximations. Since the invariant density is nonuniform, the samples quickly concentrate near the high-density region. The data-adaptive mesh therefore places smaller and more numerous cells in the region where the invariant measure has most of its mass, while the uniform mesh allocates the same resolution to all parts of the state space; see \Cref{fig:CatA}. For a fixed number of cells, this leads to a more accurate approximation of the invariant density; see \Cref{fig:CatB}. 

The main additional cost of the unstructured mesh is the need to determine cell membership for each sample. To assess this trade-off, \Cref{fig:CatC} compares the efficiency of the two approaches for a fixed number of cells, $n=400$, while varying the number of iterations $K$. The data-adaptive mesh reaches a lower-error approximation at comparable computational cost for the mesh sizes considered here. This experiment supports the use of data-adaptive unstructured meshes in the following reconstruction experiments, where the observed trajectories typically occupy only a small and highly nonuniform region of the ambient state space.
\subsection{Invariant Measure Matching}\label{IM}
In this section, we study the use of invariant measure matching as an objective function for learning the underlying dynamics. We apply the method to the Lorenz-63 and Lorenz-96 systems and compare its performance with a pointwise trajectory-matching approach.

The motivation for using invariant measures is that, for chaotic systems, pointwise trajectory errors can grow rapidly even when the learned dynamics are qualitatively correct. Thus, requiring a reconstructed trajectory to remain close to a reference trajectory at each time step can be unnecessarily restrictive. Instead, invariant measure matching compares the long-time statistical behavior of the learned system with that of the observed data. In particular, we use a finite-dimensional approximation of the invariant measure induced by a Voronoi partition of the state space. The resulting loss encourages the learned dynamics to reproduce the correct statistical distribution of states, rather than only matching short-time pointwise predictions.

\subsubsection{Clean Data}

We first consider the noise-free setting. In \Cref{fig:CleanComparisonL63IM}, we reconstruct Lorenz-63 and Lorenz-96 dynamics from clean trajectory data and compare the resulting trajectories with the ground truth. For the Lorenz-63 system we learn  the $\dot{x}$ component of the velocity, while in the Lorenz-96 case we learn the spatially dependent forcing $F$. This experiment provides a visual assessment of whether the invariant measure loss can reconstruct the qualitative structure of the underlying attractor in the absence of observation noise.

For each system, we train the model using a trajectory of length $N = 10^4$. We use $ n = 10^2$ Voronoi cells to approximate the invariant measure for Lorenz-63 and $ n  = 5\cdot 10^2$ for the Lorenz-96 system. We use the partition of unity weights based on \eqref{eq:smooth_weights2} and set the weighting parameter as $\varepsilon = 0.5$. After training for $5\cdot 10^3$ iterations, we generate reconstructed trajectories for the learned Lorenz-63 and Lorenz-96 systems and compare them with trajectories from the ground truth dynamics.

The reconstructed trajectories capture the main geometric features of the corresponding attractors. For Lorenz-63, the learned dynamics reproduce the characteristic two-lobed structure of the attractor. For Lorenz-96, the reconstructed trajectory occupies a region of state space consistent with the ground truth trajectory in the projected coordinates. These results suggest that, in the clean-data regime, invariant measure matching provides an effective objective for reconstructing the long-time statistical behavior of chaotic dynamical systems.

\begin{figure} 
    \centering
    \begin{subfigure}[t]{0.45\textwidth}
        \centering
        \includegraphics[width=\textwidth]{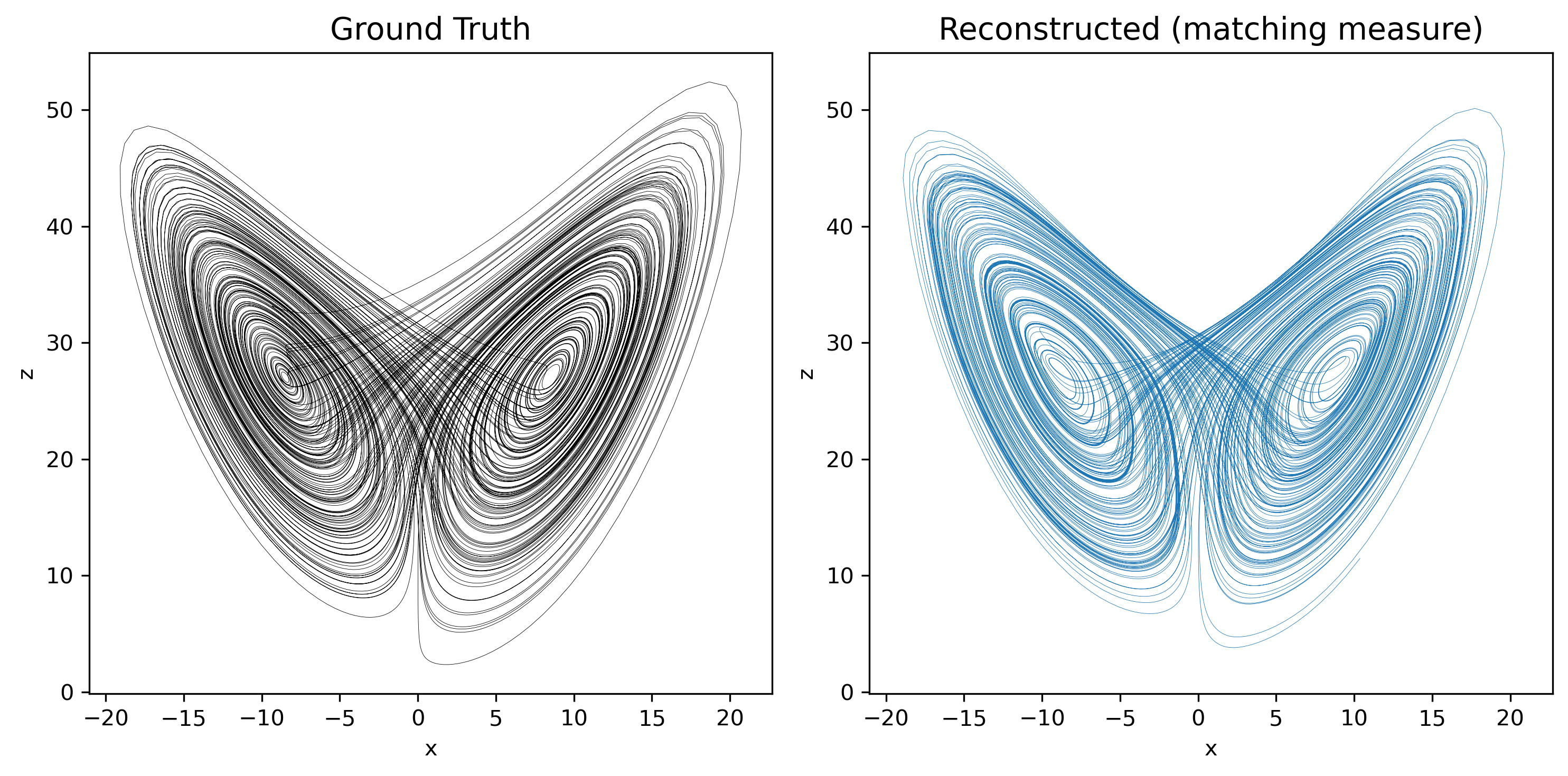}
        \caption{Lorenz-63 invariant measure learning}
        \label{fig:L63IM}
    \end{subfigure}
    \begin{subfigure}[t]{0.45\textwidth}
        \centering
        \includegraphics[width=\textwidth]{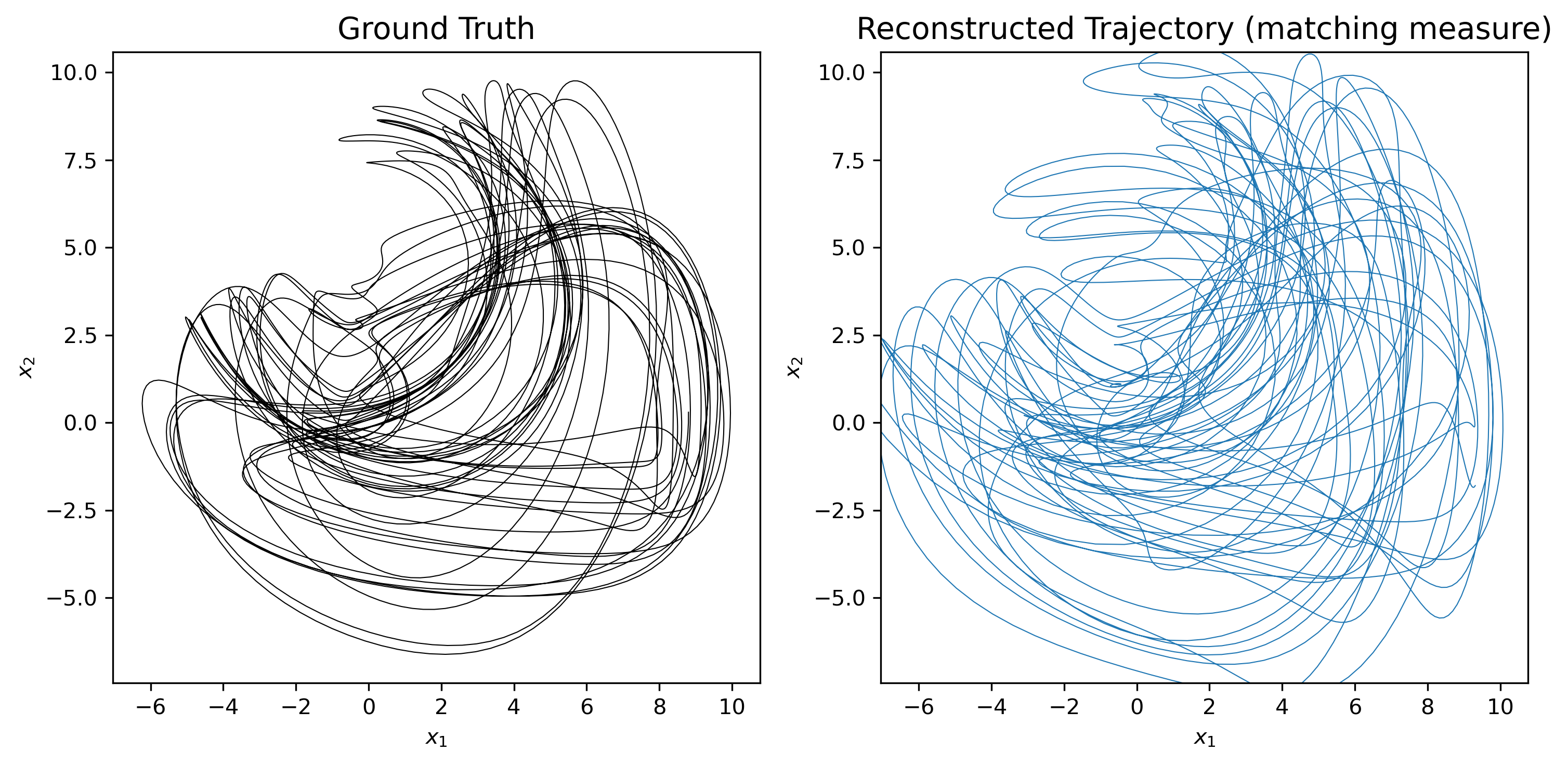}
        \caption{Lorenz-96 invariant measure learning}
        \label{fig:CleanL96IM}
    \end{subfigure}
    \caption{Two-dimensional projections of a ground truth trajectory (black) and a trajectory reconstructed by our invariant measure matching method (blue). The dynamical systems in panels (a) and (b) are Lorenz-63 and Lorenz-96, respectively.}
    \label{fig:CleanComparisonL63IM}
\end{figure}

\subsubsection{Noisy Data}

We now compare the performance of invariant measure matching and pointwise matching for learning the $\dot{x}$ component of the Lorenz-63 vector field under sparsely sampled and noisy conditions. We consider $N = 500$ noisy input-output pairs of the ground truth time-$0.05$ flow map of the Lorenz-63 system which were randomly sampled from a long trajectory. The observations are corrupted by additive Gaussian noise drawn from $\mathcal{N}(0,\sigma^2)$ with noise level $\sigma = 0.25$.
We construct $ n = 20$ Voronoi cells, and set the weighting parameter $\varepsilon = 10$ following~\eqref{eq:smooth_weights2}. We then use the forward Euler method with $dt = 0.01$ to approximate the underlying time-$0.05$ flow map. 

Training is conducted until the loss is reduced to $5\%$ of its initial value. This requires $2.5\cdot 10^3$ iterations for invariant measure matching and $2.6\cdot 10^4$ iterations for pointwise matching~\eqref{eq:least_square}. This stopping criterion ensures that both training procedures reach a comparable level of convergence. The results from one representative experiment are shown in \Cref{fig:ComparisonL63IMrand}. Our method successfully reconstructs the qualitative dynamics, while the pointwise method is strongly affected by the data-sparsity and noise.

\begin{figure}[h!] 
    \centering
    \includegraphics[width=\textwidth]{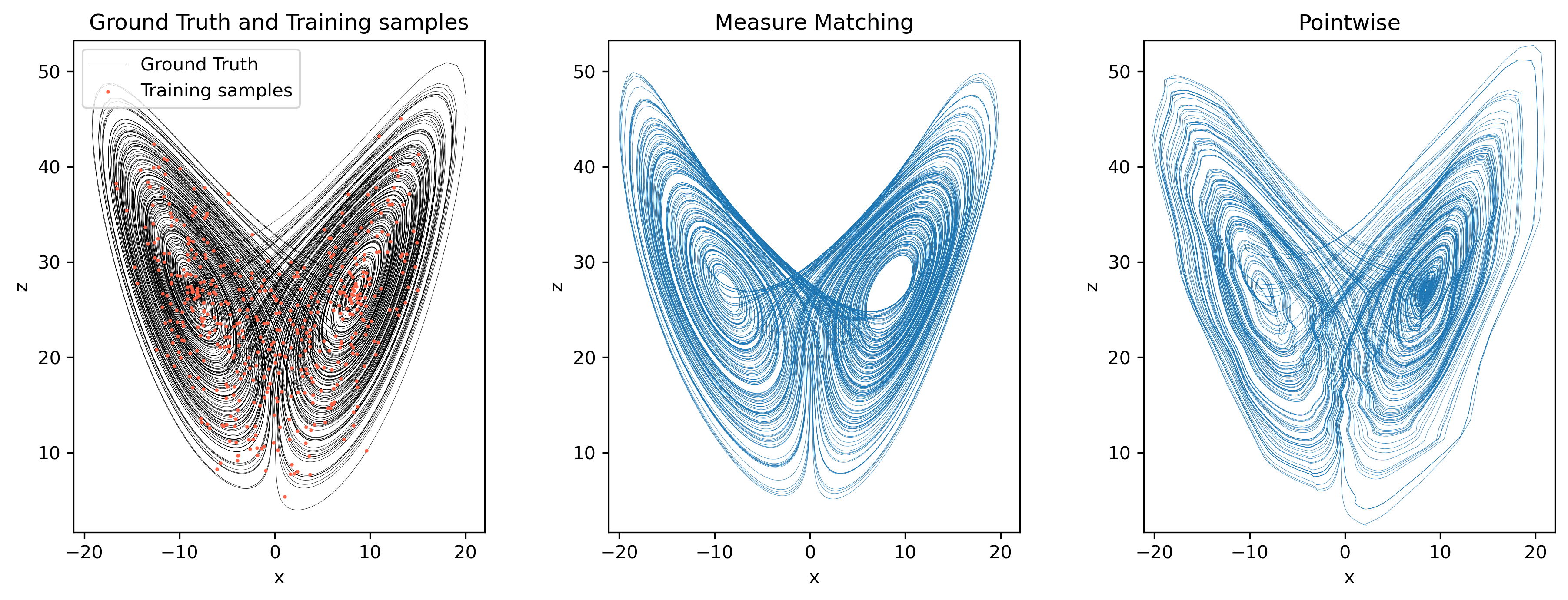}
    \caption{Two-dimensional projections of a ground truth trajectory and noisy training samples (left) and trajectories reconstructed using the proposed measure-matching method (middle) and the pointwise method (right), respectively.}
    \label{fig:ComparisonL63IMrand}
\end{figure}

We observe similar behavior for learning the spatially dependent forcing $F$ in the Lorenz-96 system defined in~\eqref{eq:lorenz96}. In this case, we consider $ N = 5\cdot 10^3$ input-output pairs of the time-$0.05$ flow map which were corrupted with a noise level of $\sigma = 0.2$. We use $n = 2\cdot 10^2$ Voronoi cells and set the weighting parameter $\varepsilon = 20$. The neural networks are trained until the loss is reduced to $25\%$ of its initial value, which requires $5\cdot 10^2$ iterations for invariant measure matching and $1.75\cdot 10^4$ iterations for pointwise matching. As shown in \Cref{fig:ComparisonL96IMrand}, invariant measure matching is more robust to noise, whereas the trajectory reconstructed by the pointwise method collapses onto a spurious periodic orbit.

\begin{figure}[h!]
    \centering
    \includegraphics[width=\textwidth]{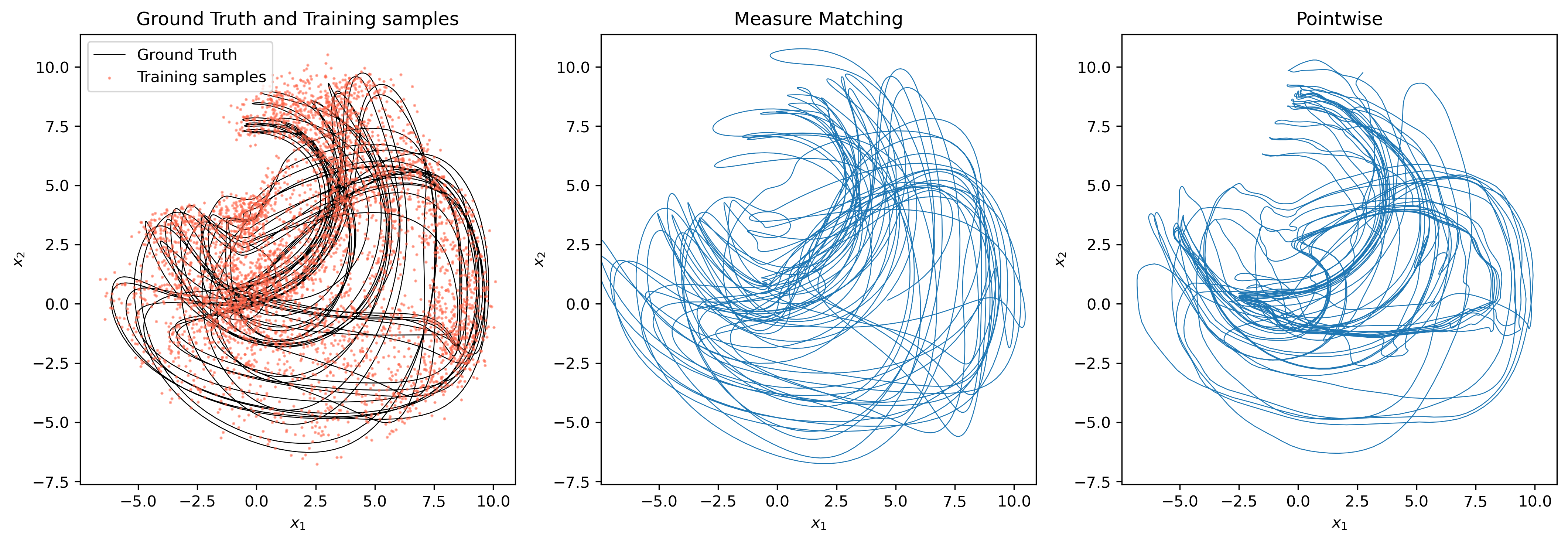}
    \caption{Two-dimensional projections of a ground truth trajectory and noisy training samples (left) and trajectories reconstructed using invariant measure matching and pointwise matching, respectively.} %
    
    \label{fig:ComparisonL96IMrand}
\end{figure}
In \Cref{tab:IM}, we repeat these experiments with different randomized neural network initializations and compare the RMSE one-step reconstruction error and the Wasserstein-2 reconstruction error for invariant measure matching and pointwise matching. For the Lorenz-63 and Lorenz-96 systems, invariant measure matching achieves smaller errors compared to pointwise matching for both the RMSE and Wasserstein-2 performance metrics. In several noisy experiments, the trajectory reconstructed by the pointwise method becomes trapped in a spurious orbit, leading to large variance in the Wasserstein-2 error.

\begin{table} 
    \centering
    \begin{tabular}{|c|c|c|c|c|}
        \hline
        System & RMSE (measure) & RMSE (pointwise)& $W_2$ (measure)& $W_2$ (pointwise)\\
        \hline

        Lorenz-63  & 0.41$\pm$0.04 & 0.74$\pm$0.07 & 9.59$\pm$0.22& 12.29$\pm$0.86 \\
        Lorenz-96  & 0.04$\pm$0.00 & 0.09$\pm$0.03 & 2.49$\pm$0.26 & 3.02$\pm$0.49\\
        \hline
    \end{tabular}
    \caption{The sample mean and standard deviation of the RMSE and $W_2$ errors for the reconstructions of Lorenz-63 and Lorenz-96 systems in Figure~\ref{fig:ComparisonL63IMrand} and Figure~\ref{fig:ComparisonL96IMrand}. The error statistics are computed across 10 different randomized realizations of neural network training. 
    }
    \label{tab:IM}
\end{table}

\subsection{Markov Matrix Matching}\label{MAT}

Unlike invariant measure matching, which compares only the long-time distribution over cells, Markov matrix matching compares transition probabilities between cells. It therefore captures not only where trajectories spend time, but also how they move through state space. 

\subsubsection{Clean Data}
We first evaluate the method in the noise-free setting. The goal is to determine whether matching cell-to-cell transition statistics is sufficient to reconstruct the qualitative dynamics of the underlying system. We consider three clean-data experiments. The first two use the same low-dimensional test problems as in the invariant measure experiments: Lorenz-63 and Lorenz-96 with $d=5$. These experiments test whether Markov matrix matching can reconstruct the qualitative dynamics in moderate dimension. The third experiment applies the method to a higher-dimensional Lorenz-96 system with $d=30$, demonstrating that the unstructured mesh construction remains feasible beyond the low-dimensional regime.

For the Lorenz-63 and $d=5$ Lorenz-96 experiments shown in \Cref{fig:CleanComparisonMAT}, we train the model using trajectories of length  $N = 10^4$. For Lorenz-63, we construct $n = 10^2$ cells and set the weighting parameter to $\varepsilon = 0.5$ following \eqref{eq:smooth_weights2}. For Lorenz-96 with $d=5$, we construct $n = 2\cdot 10^2$ cells and again set $\varepsilon = 0.5$. The Lorenz-63 experiment uses the 2-Wasserstein distance over the full transition matrices to optimize the parameters of the neural network, while the Lorenz-96 experiment uses the reduced row-wise Wasserstein distance \eqref{eq:background_M_loss2}. In both cases, we train a neural network to learn the velocity for $5\cdot 10^3$ iterations. After training, we reconstruct the corresponding three- and five-dimensional trajectories and visualize them through two-dimensional projections.

\begin{figure} 
    \centering
    \begin{subfigure}[t]{0.45\textwidth}
        \centering
        \includegraphics[width=\textwidth]{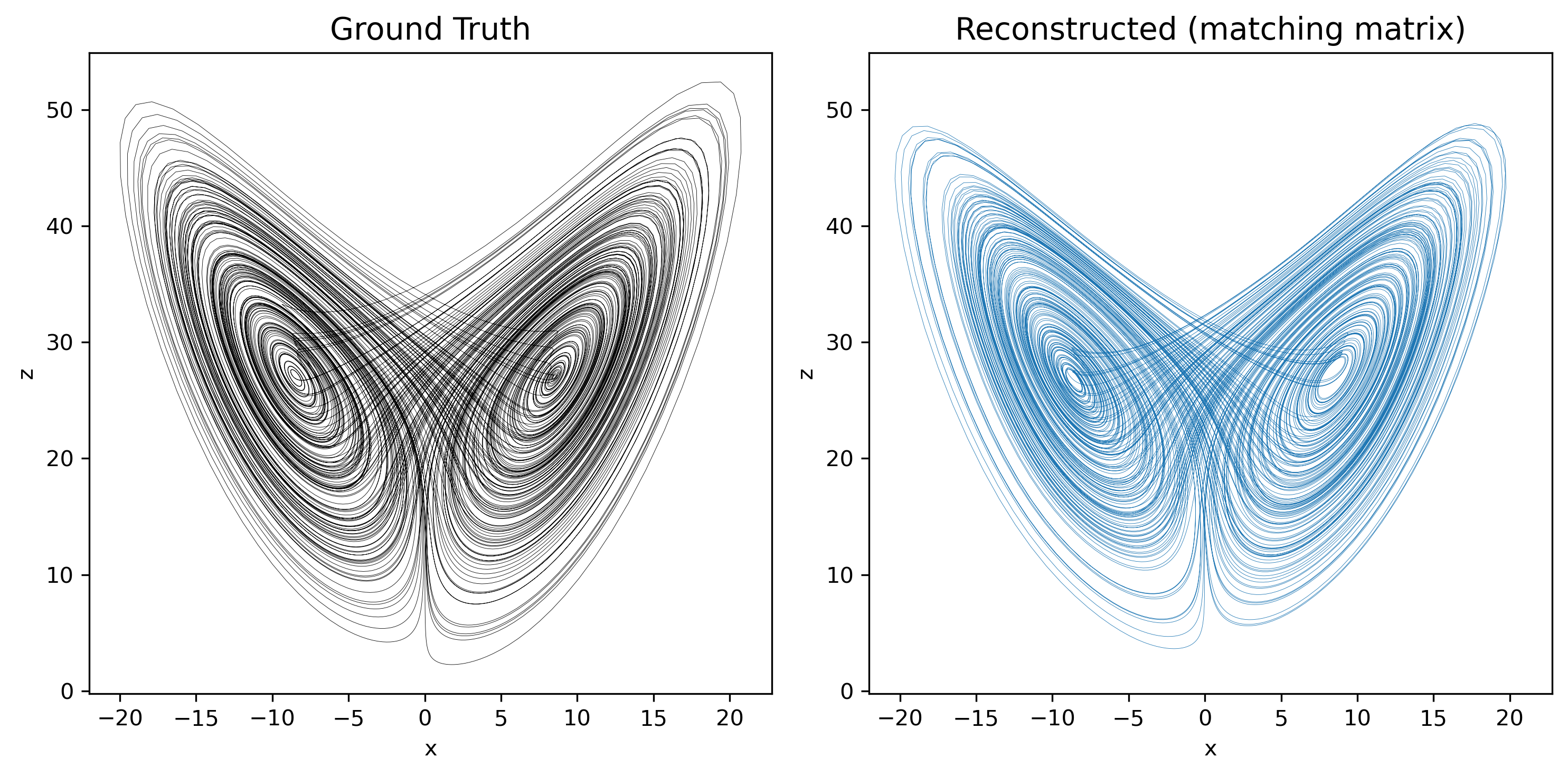}
        \caption{Lorenz-63 Markov matrix learning}
        \label{fig:CleanL63MAT}
    \end{subfigure}
    \begin{subfigure}[t]{0.45\textwidth}
        \centering
        \includegraphics[width=\textwidth]{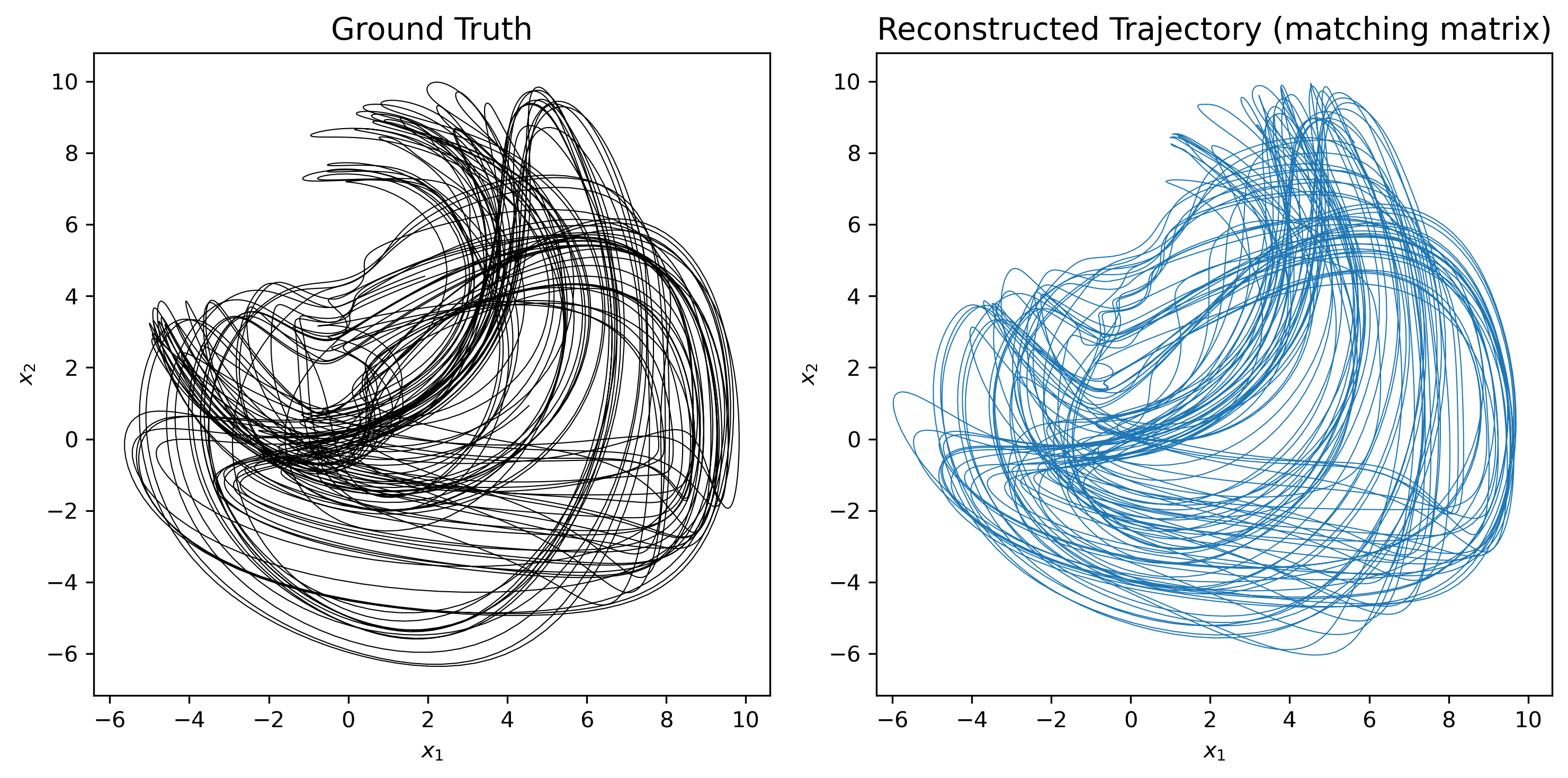}
        \caption{Lorenz-96 Markov matrix learning}
        \label{fig:CleanL96MAT}
    \end{subfigure}
    \caption{Two-dimensional projections of a ground truth trajectory (black) and a trajectory reconstructed by our Markov matrix matching method (blue). The dynamical systems in panels (a) and (b) are Lorenz-63 and Lorenz-96 with $d=5$, respectively.}
    \label{fig:CleanComparisonMAT}
\end{figure}

The results in \Cref{fig:CleanComparisonMAT} show that Markov matrix matching accurately reconstructs the qualitative structure of both systems. For Lorenz-63, the reconstructed trajectory reproduces the characteristic two-lobed attractor. For Lorenz-96 with $d=5$, the reconstructed trajectory remains consistent with the ground truth dynamics in the projected coordinates. These results indicate that matching transition statistics between cells is sufficient to reconstruct the main geometric features of the dynamics in the clean-data setting.

We next test the scalability of the method on a higher-dimensional Lorenz-96 system. Specifically, we consider \eqref{eq:lorenz96} with $d=30$ and seek to reconstruct the first component of the velocity field using a neural network parameterization. For this example, we construct weighting functions following~\eqref{eq:smooth_weights} with $\varepsilon = 5$ and use $n=2\cdot 10^2$ unstructured cells to construct the Markov transition matrices. Moreover, we use the Frobenius distance to compare Markov transition matrices during optimization. The results in \Cref{fig:L96} show close agreement between the reconstructed and ground truth systems after training. The long-time projected trajectory produced by the learned model follows the same qualitative structure as the true trajectory, while the short-time visualization of all $30$ state variables shows that the reconstructed dynamics reproduce the spatiotemporal behavior of the Lorenz-96 system. This experiment demonstrates that the proposed Markov matrix matching framework, combined with an unstructured data-adaptive mesh, can be effectively applied to high-dimensional dynamical systems.

\begin{figure}[h!] 
    \centering
    \subfloat[Three-dimensional projections of a ground truth trajectory (left) and a trajectory from our reconstructed model (right).]{
        \includegraphics[width=0.8\linewidth]{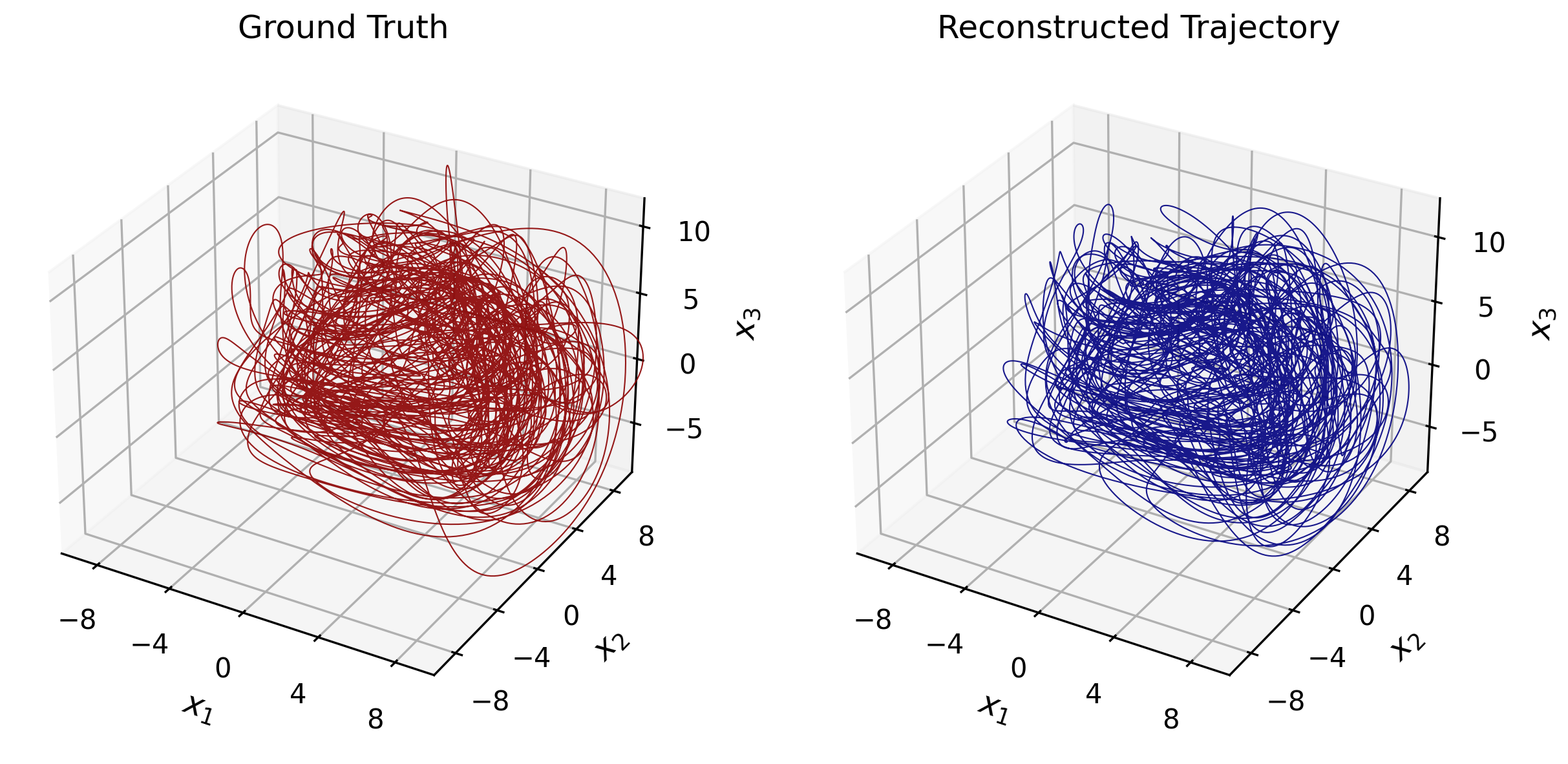}
        \label{fig:L96a}
    }
    \\
    \vspace{.2cm}
    \subfloat[Visualizations of short trajectories from the ground truth system (left) and our reconstructed model (right). Each row corresponds to one state variable, and color indicates its value at a given time.]{
        \includegraphics[width=0.8\linewidth]{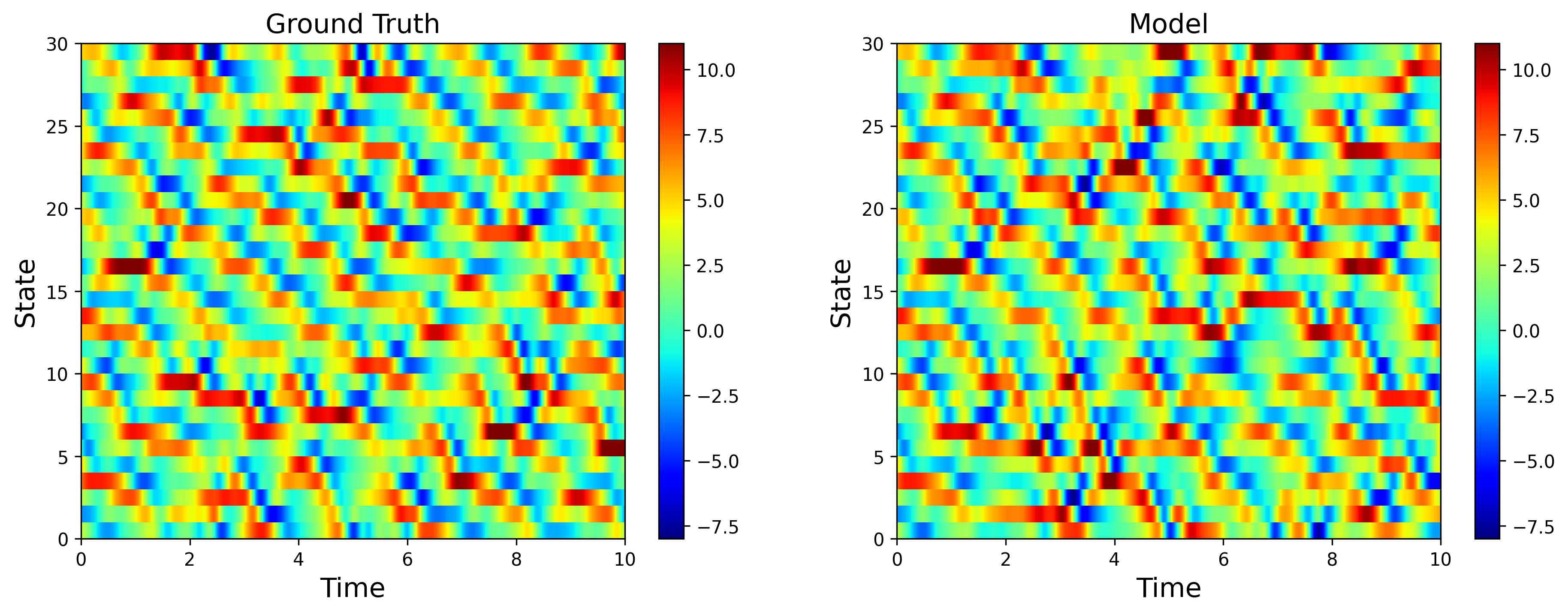}
        \label{fig:L96b}
    }
    \caption{Reconstruction of the first velocity component $\dot{x}_1=v_\theta(x)$ for the $30$-dimensional Lorenz-96 system. Panel (a) compares projected long trajectories from the ground truth and reconstructed systems. Panel (b) compares the time evolution of all $30$ state variables over a short trajectory.}
    \label{fig:L96}
\end{figure}

\subsubsection{Noisy Data}

We next evaluate Markov matrix matching in the presence of additive white noise. We consider two sampling regimes. First, we train on a single trajectory segment corrupted by noise. Second, we consider a sparse collection of noisy input-output pairs randomly sampled from a long trajectory and separated by a longer temporal window. These experiments test whether the transition-based objective remains stable when pointwise trajectory information is corrupted by noise. 

We begin by considering a trajectory of the Lorenz-63 system of length  $N = 5\cdot 10^3$ corrupted by Gaussian noise with level $\sigma = 0.3.$ We construct $n = 50$ Voronoi cells and set the weighting parameter as  $\varepsilon = 1$ following \eqref{eq:smooth_weights2}. We compare Markov matrix matching with pointwise matching \eqref{eq:least_square} and SINDy in order to assess long-time reconstruction quality in a noisy environment. The neural networks are trained until the loss is reduced to $15\%$ of its initial value. For this experiment, we use the Wasserstein distance to compare transition matrices as the objective in Markov matrix matching. This takes $10^3$ iterations for matrix matching and $4.9\cdot 10^4$ iterations for pointwise matching. We also fit a model using SINDy with a polynomial basis of degree up to 6. As shown in \Cref{fig:ComparisonL63}, matrix matching remains robust to white noise, while the trajectory reconstructed by the pointwise method becomes trapped in a spurious periodic orbit and the trajectory reconstructed by SINDy diverges rapidly. 

\begin{figure}[h!]
    \centering
    \includegraphics[width=\textwidth]{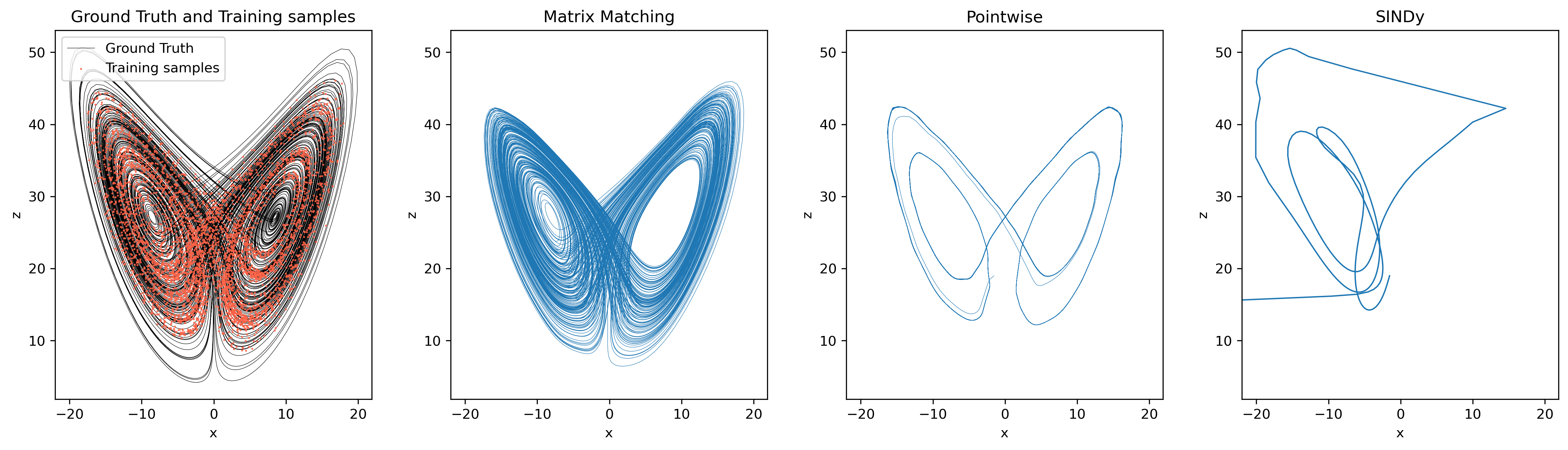}
    \caption{Two-dimensional projections, from left to right, of the ground truth trajectory with noisy training samples, the trajectory reconstructed by our Markov matrix matching method, the trajectory reconstructed by the pointwise method, and the trajectory reconstructed by SINDy.}
    \label{fig:ComparisonL63}
\end{figure}

We observe similar behavior for the Lorenz-96 system with $d=5$ and noise level $\sigma = 0.05$. In this experiment, we consider a trajectory of length $N = 5\cdot 10^3$, we construct $n = 2\cdot 10^2$ cells, and we set $\varepsilon = 10$ following \eqref{eq:smooth_weights2}. We consider matrix matching with the reduced row-wise Wasserstein loss \eqref{eq:background_M_loss2}. The neural networks are trained until the loss is reduced to $5\%$ of its initial value. This requires $5\cdot 10^2$ iterations for our method and $6.05\cdot 10^4$ iterations for the pointwise method. We also fit a model using SINDy with a polynomial basis of degree up to 5. As shown in \Cref{fig:ComparisonL96}, matrix matching is robust to noise, while the trajectory reconstructed by the pointwise method quickly becomes trapped in a spurious periodic orbit. The trajectory reconstructed by SINDy again diverges.

\begin{figure}[h!]
    \centering
    \includegraphics[width=\textwidth]{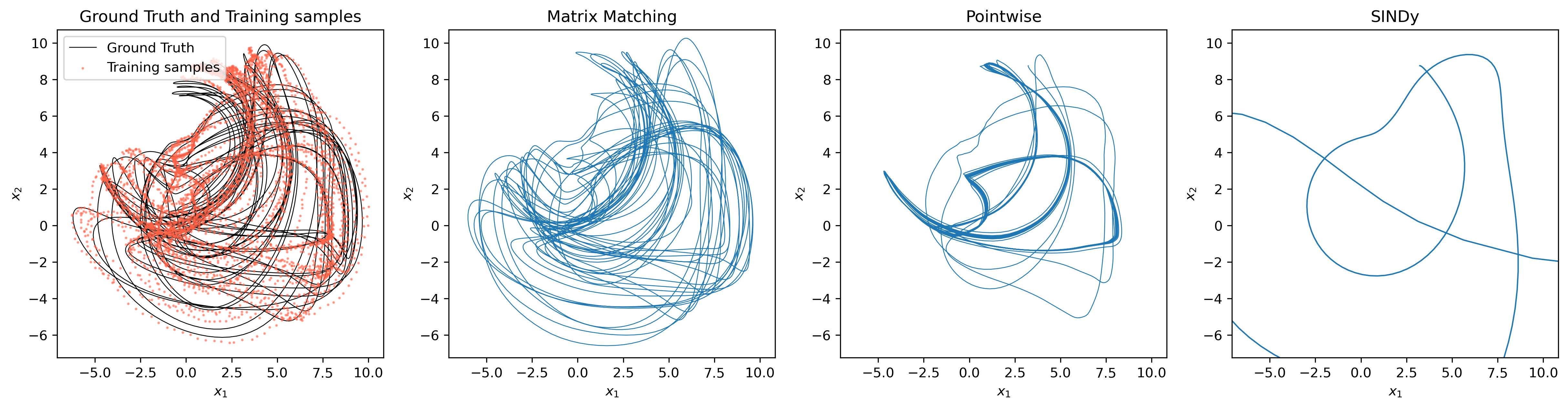}
    \caption{Two-dimensional projections of the ground truth trajectory and noisy training samples (left), together with trajectories reconstructed using our Markov matrix matching method, the pointwise method, and SINDy.}
    \label{fig:ComparisonL96}
\end{figure}

We next consider a sparser dataset from the Lorenz-63 system to further test the methods. In particular, we observe $N = 5\cdot 10^2$ input-output pairs of the time-$0.05$ flow map randomly sampled from a long trajectory and corrupted with noise level $\sigma = 0.5$. We construct $n = 20$ cells, and set $\varepsilon = 2$. For matrix matching we consider the Wasserstein distance for comparing the full transition matrices. The neural networks are trained until the loss is reduced to $2\%$ of its initial value, which takes $3.5\cdot 10^3$ iterations for matrix matching and $9.5\cdot 10^3$ iterations for pointwise matching. \Cref{fig:ComparisonL63MATrand} shows the result from one representative experiment. Matrix matching reconstructs the qualitative structure of the dynamics, whereas the dynamics arising from the pointwise matching rapidly collapse into a small region of state space.

\begin{figure}[h!]
    \centering
    \includegraphics[width=\textwidth]{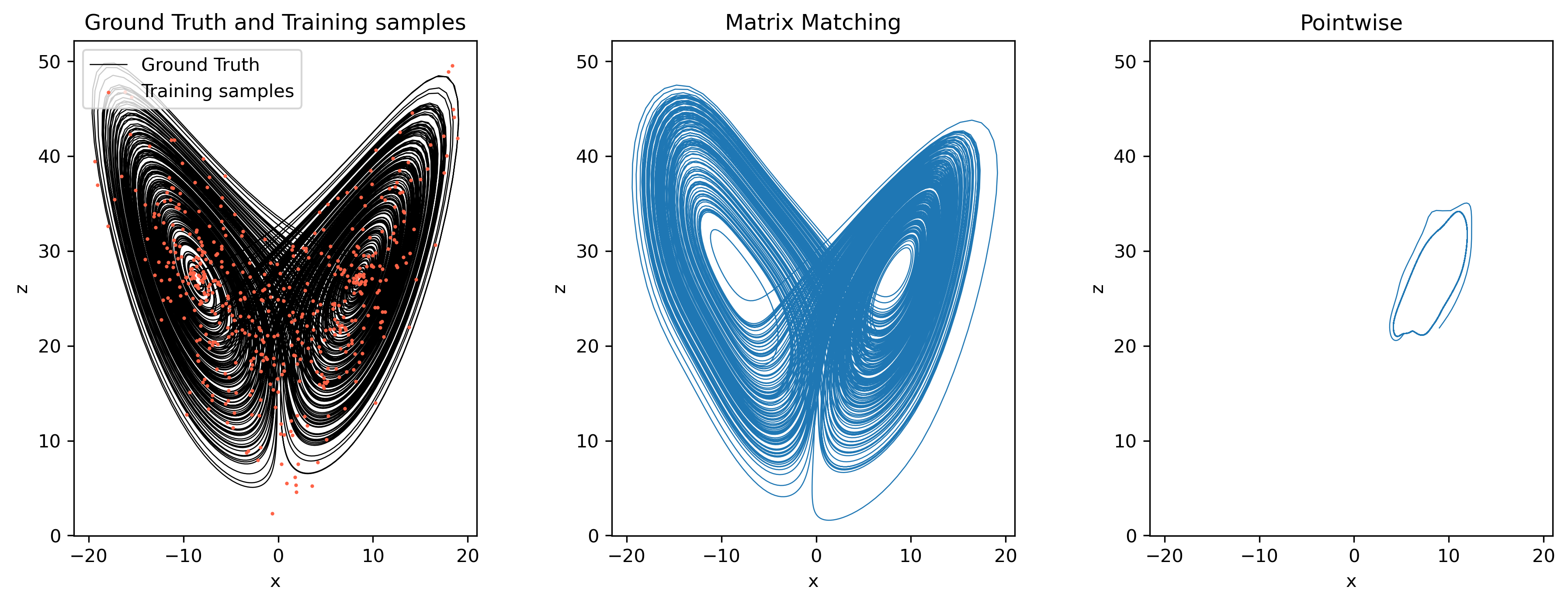}
    \caption{Two-dimensional projections of the ground truth trajectory and noisy training samples (left), together with trajectories reconstructed using our Markov matrix matching method and the pointwise method.}
    \label{fig:ComparisonL63MATrand}
\end{figure}

For the Lorenz-96 system with $d = 5$, we train on $N = 5\cdot 10^3$ input-output pairs of the time-$0.05$ flow map randomly sampled from a long trajectory and corrupted with noise level $\sigma = 0.2$. We construct $n = 10^2$ cells and set $\varepsilon = 10$ following \eqref{eq:smooth_weights2}. The neural networks are trained until the loss is reduced to $2\%$ of its initial value, requiring $10^3$ iterations for matrix matching and $2\cdot 10^3$ iterations for pointwise matching. As shown in \Cref{fig:ComparisonL96MATrand}, matrix matching successfully reconstructs the dynamics, while the pointwise reconstruction is visibly distorted by noise.

\begin{figure} 
    \centering
    \includegraphics[width=\textwidth]{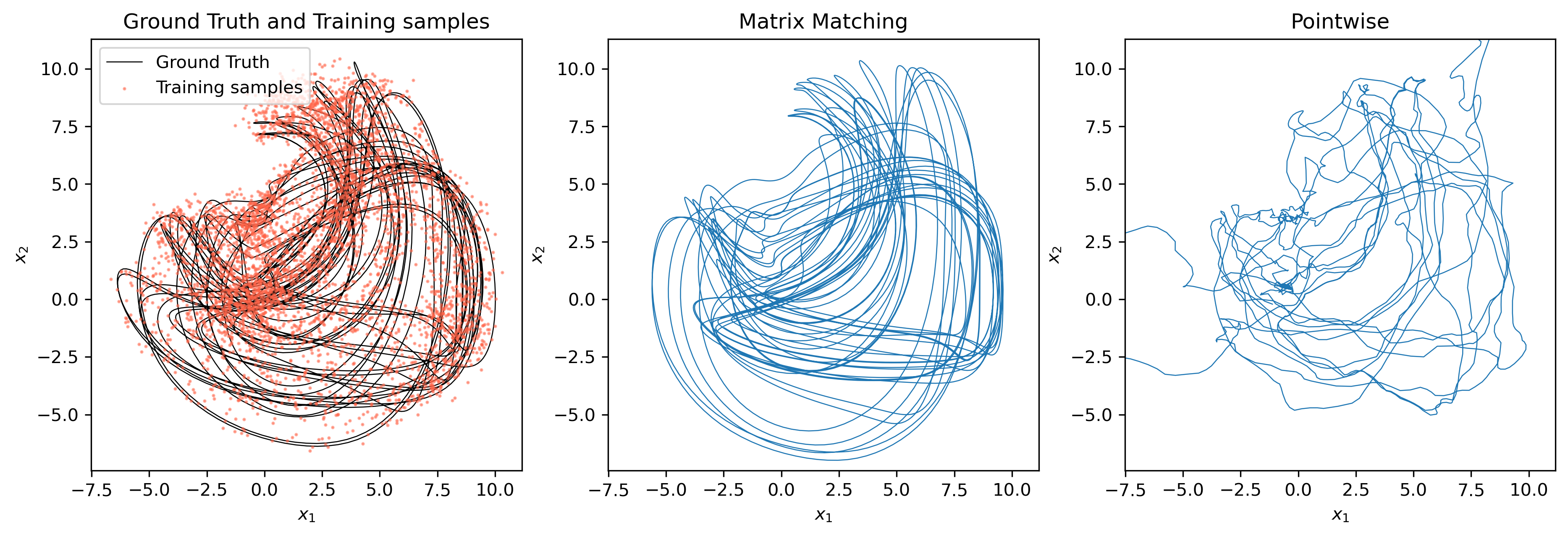}
    \caption{Two-dimensional projections of the ground truth trajectory and noisy training samples (left), together with trajectories reconstructed using our Markov matrix matching method and the pointwise method.}
    \label{fig:ComparisonL96MATrand}
\end{figure}

Finally, \Cref{tab:mat} compares the one-step RMSE and the $W_2$ reconstruction error for our Markov matrix matching approach and the pointwise method. Across both systems, matrix matching achieves smaller mean errors and smaller standard deviations. In several noisy experiments, the pointwise reconstruction becomes trapped near a spurious fixed point or periodic orbit, leading to large variance in the $W_2$ error.

\begin{table} 
    \centering
    \begin{tabular}{|c|c|c|c|c|}
        \hline
        System & RMSE (matrix) & RMSE (pointwise) & $W_2$ (matrix) & $W_2$ (pointwise)\\
        \hline
        Lorenz-63  & $0.51\pm 0.02$ & $0.60\pm 0.02$ & $5.38\pm 0.35$ & $19.65\pm 0.88$ \\
        Lorenz-96  & $0.09\pm 0.01$ & $0.34\pm 0.01$ & $2.78\pm 0.68$ & $6.53\pm 3.65$ \\
        \hline
    \end{tabular}
    \caption{RMSE one-step reconstruction errors and $W_2$ reconstruction errors for the Lorenz-63 and Lorenz-96 experiments shown in \Cref{fig:ComparisonL63MATrand,fig:ComparisonL96MATrand}. The error statistics are computed across 10 different randomized realizations of neural network training.}
    \label{tab:mat}
\end{table}

\subsection{Application to NOAA SST Dataset}\label{SST}

We next apply our methods to a real-world forecasting problem using the NOAA sea surface temperature (SST) dataset \cite{reynolds2002improved}. The SST snapshots are high-dimensional spatial fields, so we first reduce the dimension of the data before learning the dynamics. After constructing a proper orthogonal decomposition (POD) coordinate system of the SST fields, we train on the coefficients corresponding to the first $449$ monthly snapshots, and we use the following $50$ months to assess predictive ability. We compare the results of training using the Markov matrix matching \eqref{eq:background_M_loss} and invariant measure matching~\eqref{eq:background_IM_loss} objectives.

 Let $s_j \in \mathbb{R}^m$ denote the SST field at month $j$, where $m\approx 4.5\cdot 10^4$ is the number of spatial grid points. We apply a POD reduction to the snapshots and retain the first four dominant modes. Thus, each SST field is approximated by
\[
    s_j \approx \bar{s} + \Phi a_j,
\]
where $\bar{s}\in\mathbb{R}^m$ is the temporal mean, $\Phi \in \mathbb{R}^{m \times 4}$ contains the first four POD modes, and $a_j \in \mathbb{R}^4$ is the corresponding vector of POD coefficients. This representation converts the original high-dimensional SST forecasting problem into a four-dimensional time-series prediction problem. Before training the neural networks, we apply $z$-score normalization to the POD coefficients. Specifically, each coefficient is shifted and scaled using the mean and standard deviation computed from the data, so that the normalized coefficients have zero mean and unit variance. The POD reduction and normalization statistics are computed once from the full SST sequence, yielding a fixed coordinate system for training and testing models.

For this experiment we construct the weighting function following \eqref{eq:smooth_weights} with $\varepsilon = 2$ and we use $n = 30$ cells. For the Markov matrix matching objective, we minimize the Frobenius distance between transition matrices induced by the observed data and by the learned dynamics, while for the invariant measure matching objective, we minimize the $L^2$ distance between the simulated and observed invariant measures. Both methods use a fully connected neural network with two hidden layers of 100 nodes and hyperbolic tangent activation function to parameterize the velocity in the rescaled POD coefficient space. The flow map is approximated using five steps of a forward Euler scheme with $dt = 0.01$.  The models are trained until the loss reduces below $0.5\%$ of its initial value, which takes $3.5\cdot 10^3$ iterations for the matrix matching objective and $2.5\cdot 10^3$ iterations for the invariant measure matching objective.

After training, we initialize the model at the first POD coefficient in the testing dataset and simulate the learned four-dimensional dynamical system forward for $49$ months. This produces predicted POD coefficients $\hat{a}_j$ for the test period. We then transform the predictions $\hat{a}_j$ back to the original SST space using the POD reconstruction formula
\[
    \widehat{s}_j = \bar{s} + \Phi \widehat{a}_j.
\]
 We compare these reconstructed SST fields for both matrix and invariant measure matching with the true NOAA SST snapshots.

    \begin{figure}[h!]
    \centering
    \includegraphics[width=\textwidth]{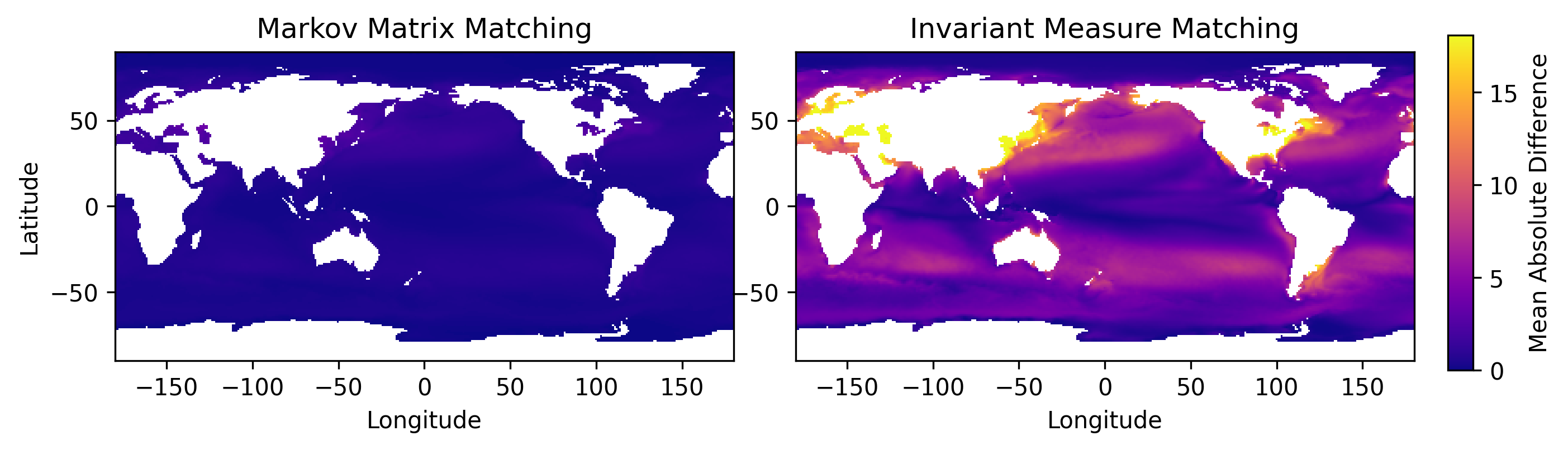}
    \caption{Spatial distributions of the time-averaged absolute error for the Markov matching SST reconstruction (left) and the reconstruction from the invariant measure matching method (right).}
    \label{fig:SST_MAE_1}
    \end{figure}
    \begin{figure} 
    \centering
    \includegraphics[width=\textwidth]{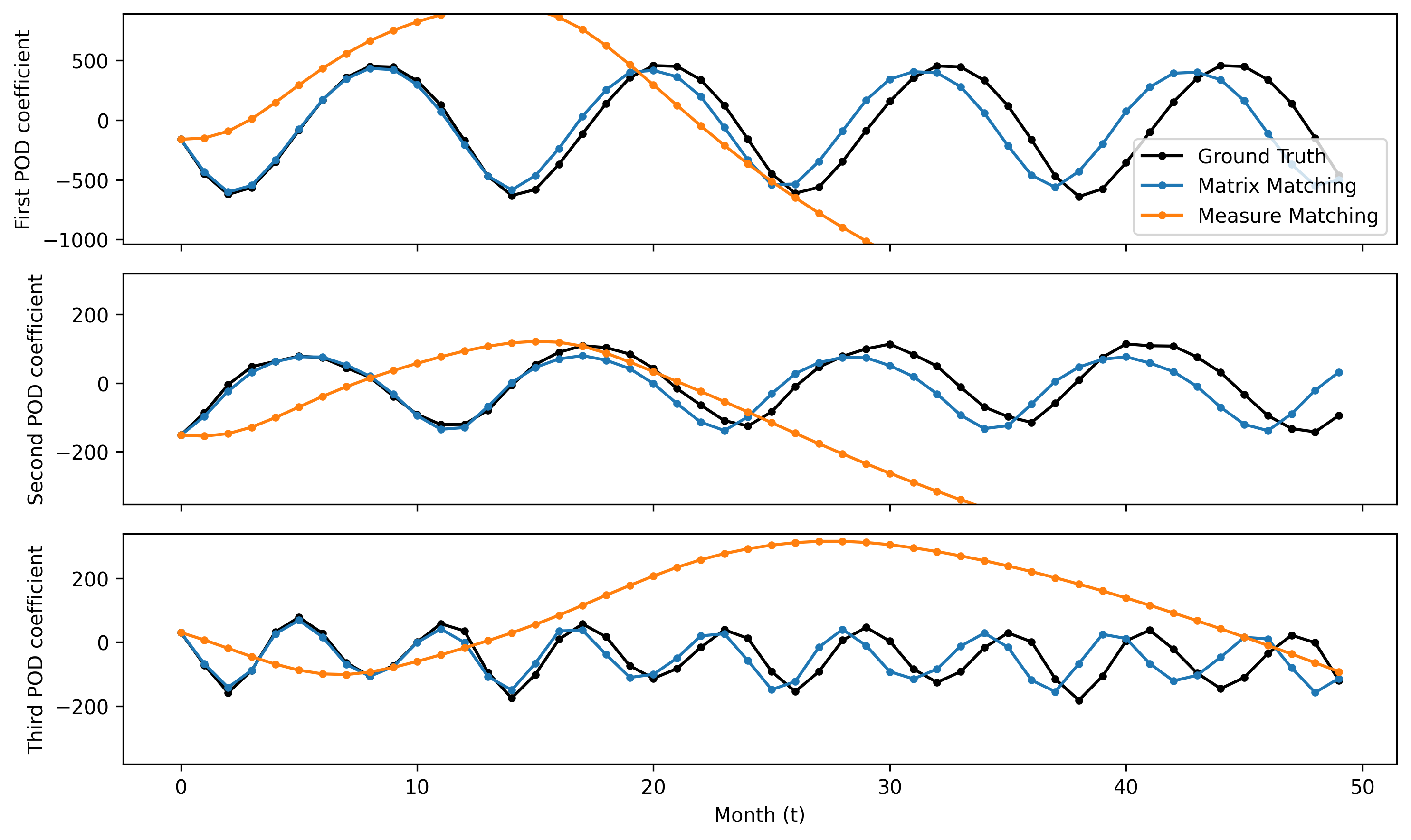}
    \caption{Comparison of the evolution of the first three components. The black curve is the ground truth. The blue and orange curves are the reconstructions using our matrix matching and invariant measure matching. }
    \label{fig:SST_MAE}
    \end{figure}
    
Figure~\ref{fig:SST_MAE_1} shows the spatial distribution of the time-averaged absolute reconstruction error over the $50$-month testing window. The left panel corresponds to our Markov matrix matching method, while the right panel corresponds to the invariant measure matching. These plots compare where each method accumulates error in the reconstructed SST field and provide a spatially resolved assessment of forecast accuracy. In Figure~\ref{fig:SST_MAE}, we compare the predicted time evolution of the first three POD coefficients. The black curve denotes the ground truth, while the blue and orange curves denote the reconstructions obtained using matrix matching and invariant measure matching, respectively. The matrix matching method produces a reliable forecast of the unseen testing dynamics, while the invariant measure matching  fails to accurately predict the POD coefficients. Across ten trials with randomized neural network initializations we find that the one-step RMSE error for prediction within POD coordinates is  $117.70 \pm 13.70$ for invariant measure matching and $25.97\pm 2.63$ for Markov matrix matching. Since the POD coefficients encode the dominant temporal variability of the SST field, accurate prediction of these components is essential for producing reliable full-field reconstructions.

These findings are not unexpected, as the Markov matrix matching objective encodes significantly more information compared to the invariant measure matching objective. When constraining the learned dynamics using only the invariant measure, one can only enforce that the model's long-term statistical behavior matches the data. In contrast, Markov matrix matching captures short-time transient behaviors described through transition probabilities between different regions of state space. In Section \ref{IM}, we restricted the parametric model class when using the invariant measure to infer the dynamics to ensure unique recovery. The poor performance of invariant measure matching in Figures~\ref{fig:SST_MAE_1} and~\ref{fig:SST_MAE} suggests that without prior knowledge of the functional form of the underlying dynamics one should use the Markov matrix matching objective for trajectory forecasting. In situations where the data are too sparse and noisy to reliably estimate finite-time transition probabilities, invariant measure matching is expected to provide additional benefits over Markov matrix matching.

\section{Conclusions}\label{sec:conclusion}
In this work, we introduced a transition-statistics framework for learning dynamical systems from trajectory data. Rather than enforcing pointwise agreement between observed and simulated trajectories, the method compares how probability mass moves through state space over a finite time interval. This Eulerian perspective is well-suited to noisy, sparsely sampled, and chaotic systems, where trajectory-level losses can be unstable or overly sensitive to small errors.

The approach is based on data-adaptive approximations of the Perron--Frobenius operator. We construct Voronoi partitions from observed trajectory data, concentrating resolution in regions visited by the dynamics. On this partition, we build a regularized Ulam-type transition matrix by replacing hard cell indicators with nonnegative partition-of-unity weights. This preserves the stochastic structure of the Markov matrix while making its entries differentiable with respect to the parameters of a learned vector field.

Within this framework, we considered two learning objectives: matching invariant measures and matching full Markov transition matrices. Invariant-measure matching compares the long-time statistical behavior of the observed and learned systems, while Markov matrix matching also captures finite-time transport between regions of state space. Our theoretical analysis establishes convergence of the regularized transition matrix in an appropriate limit and provides a finite-sample error analysis motivating data-adaptive partition construction.

Numerical experiments on the Lorenz-63 and Lorenz-96 systems show that transition-statistics objectives can produce more stable long-time reconstructions than pointwise trajectory matching, particularly under noise and coarse temporal sampling. Moreover, we compare Markov matrix matching and invariant measure matching on a reduced-order NOAA sea surface temperature forecasting problem, demonstrating that Markov matrix matching more accurately captures transient dynamics without prior knowledge of the functional form of the underlying velocity. Future work includes developing improved adaptive partitions for high-dimensional systems, scalable discrepancies between large transition matrices, and extensions to partially observed, stochastic, and irregularly sampled dynamics.

\section*{Acknowledgments}
J.~B.-G.~was supported in part by a fellowship award under contract FA9550-21-F-0003 through the National Defense Science and Engineering Graduate (NDSEG) Fellowship Program, sponsored by the Air Force Research Laboratory (AFRL), the Office of Naval Research (ONR) and the Army Research Office (ARO), and ONR under award N00014-24-1-2088. Y.~Y.~was supported in part by the National Science Foundation under award DMS-2409855 and by ONR under award N00014-24-1-2088.

This work was initiated and supported in part by the Cornell University Department of Mathematics through its Summer 2023 Research Experiences for Undergraduates (REU) program. The authors thank Robert Martin for helpful discussions and insights during the early stages of this project.

\end{document}